\documentclass[a4paper]{article}

\usepackage[pages=all, color=black, position={current page.south}, placement=bottom, scale=1, opacity=1, vshift=5mm]{background}

\usepackage[margin=1.2in]{geometry} 

\usepackage{amsmath}
\usepackage{amsthm}
\usepackage{amssymb}
\usepackage[title]{appendix}
\usepackage[utf8]{inputenc}
\usepackage{hyperref}
\usepackage{graphicx}
\usepackage{verbatim,xcolor}
\usepackage{bm}
\usepackage{multirow}
\usepackage{algorithmicx}
\usepackage[ruled]{algorithm}
\usepackage{algpseudocode}
\usepackage[symbol]{footmisc}
\usepackage{float, subcaption}
\usepackage{todonotes}
\usepackage{graphicx, color}
\usepackage{algorithm, algpseudocode}
\usepackage{mathrsfs}
\usepackage{lipsum}
\usepackage[sort&compress,numbers,square]{natbib}

\theoremstyle{plain}
\newtheorem{theorem}{Theorem}[section]

\newtheorem{lemma}[theorem]{Lemma}

\theoremstyle{definition}

\newtheorem{example}[theorem]{Example}
\newtheorem{remark}[theorem]{Remark}

\numberwithin{equation}{section}

\newcommand{\bx}{{\bf x}}

\newcommand{\bv}{{\bf v}}

\newcommand{\bw}{{\bf w}}

\newcommand{\Th}{\mathcal{T}_h}

\newcommand{\norm}[1]{\lVert #1\rVert}

\newcommand{\trinorm}[1]{{\vert\kern-0.25ex\vert\kern-0.25ex\vert #1 \vert\kern-0.25ex\vert\kern-0.25ex\vert}}

\newcommand{\LT}{{L^2(\Omega)}}
\newcommand{\Ho}{{H^1_0(\Omega)}}
\newcommand{\HO}{{H^1(\Omega)}}
\newcommand{\eps}{\varepsilon}
\newcommand{\cB}{\mathcal{B}}
\newcommand{\dx}{{\rm d}\mathbf{x}}

\DeclareMathOperator*{\argmin}{argmin}

\title{A monotone finite element method for an elliptic distributed optimal control problem with a convection-dominated state equation
}
\author{SeongHee Jeong,\thanks{School of Computing and Data Science, Wentworth Institute of Technology, Boston, MA 02115 (\texttt{jeongs2@wit.edu})}
\and Seulip Lee,\thanks{Department of Mathematics, Tufts University, Medford, MA 02155 (\texttt{seulip.lee@tufts.edu})}
\and Sijing Liu,\thanks{Department of Mathematical Sciences,  Worcester Polytechnic Institute, Worcester, MA 01609 (\texttt{sliu13@wpi.edu})}
}

\date{
}

\begin{document}
\maketitle
    
\begin{abstract}
We propose and analyze a monotone finite element method for an elliptic distributed optimal control problem constrained by a convection-diffusion-reaction equation in the convection-dominated regime. 
The method is based on the edge-averaged finite element (EAFE) scheme, which is known to preserve the discrete maximum principle for convection-diffusion problems. 
We show that the EAFE discretization inherits the monotonicity property of the continuous problem and consequently preserves the desired-state bounds at the discrete level, ensuring that the numerical optimal state remains stable and free of nonphysical oscillations. 
The discrete formulation is analyzed using a combination of the EAFE consistency result and a discrete inf-sup condition, which together guarantee well-posedness and yield the optimal convergence order. 
Comprehensive numerical experiments are presented to confirm the theoretical findings and to demonstrate the robustness of the proposed scheme in the convection-dominated regimes.

\end{abstract}


\section{Introduction}
\label{sec:intro}

In this work, we consider an elliptic distributed optimal control problem constrained by a convection-diffusion-reaction equation in the convection-dominated regime. Elliptic distributed optimal control problems and their numerical approximations have been extensively studied in the literature \cite{becker2007optimal,brenner2020multigrid,brenner2021p1,liu2025robust,brenner2024c0,jeong2025optimal}, with numerous applications in engineering, aeronautics, aerodynamics, and related fields, see \cite{Tro,Lions} for comprehensive discussions. Among these problems, convection-dominated optimal control problems are particularly challenging, as strong convective effects substantially influence both the analytical properties and the numerical approximation. Optimal control problems constrained by convection-dominated equations have been investigated in \cite{leykekhman2012local,heinkenschloss2010local,liu2024multigrid,becker2007optimal}, which revealed essential differences between solving a single convection-dominated equation and solving an optimal control problem constrained by such an equation.

A key analytical observation in this work is the role of the maximum principle, one of the fundamental properties of elliptic partial differential equations (PDEs). 
This principle is guaranteed by a sufficient condition known as the monotonicity property, which asserts that the inverse of the elliptic operator is nonnegative. 
Although this structure is inherent at the continuous level, it has not always been emphasized in the optimal control literature, particularly in the convection-dominated setting. 
When the monotonicity property is considered in the saddle-point system associated with the optimal control problem, the nonnegativity of the operator’s inverse induces a stability relation between the desired state $y_d$ and the optimal state $\bar{y}$:
\begin{equation*}
\begin{cases}
0 \leq \bar{y}(\mathbf{x}) \leq y_d(\mathbf{x}) & \text{for } y_d(\mathbf{x}) \geq 0,\ \forall\mathbf{x}\in\Omega,\\[0.7ex]
y_d(\mathbf{x}) \leq \bar{y}(\mathbf{x}) \leq 0 & \text{for } y_d(\mathbf{x}) \leq 0,\ \forall\mathbf{x}\in\Omega.
\end{cases}
\end{equation*}
These desired-state bounds are not additional constraints imposed in the optimization formulation; rather, they arise naturally from the intrinsic stability of the elliptic operator and the coupled structure of the PDE–constrained optimal control problem, in which the state and adjoint variables are linked through the saddle-point formulation. 
As a consequence, the optimal state inherits the same monotone behavior as the underlying PDE, leading to stable and physically consistent solution behavior even in convection-dominated regimes.

From a numerical perspective, convection-dominated equations are well known to pose significant challenges in the design of stable and accurate numerical schemes.
The presence of sharp boundary layers near the outflow boundary often leads to spurious oscillations in the numerical solution if no special stabilization is employed \cite{rooscdr,morton1995numerical,johnson2012numerical,knabner2003numerical}. 
To address these difficulties, a wide range of stabilization techniques within the finite element framework have been proposed, including the upwind/Petrov-Galerkin (UPG) methods \cite{brooks1982streamline,da2021supg,bacuta2025convergence}, the edge-averaged schemes \cite{xu1999monotone,adler2023stable,cao2025edge}, local projection methods \cite{knobloch2009local,li2021local}, the algebraic flux corrections (AFC) method \cite{barrenechea2018unified}, and discontinuous Galerkin (DG) formulations \cite{lesaint1974finite,ayuso2009discontinuous,arnold2002unified,arnold1982interior}, among many others \cite{rooscdr}. 
The same numerical difficulties arise in optimal control problems governed by convection-dominated equations. However, the associated saddle-point system consists of two coupled convection-dominated equations with opposite convection directions, which makes the design of stable numerical schemes even more delicate.
As demonstrated in \cite{leykekhman2012local,heinkenschloss2010local}, boundary layers can induce nonphysical oscillations that propagate into the interior of the domain when Dirichlet boundary conditions are imposed strongly, even in the presence of stabilization. DG methods are effective in mitigating this issue, as they impose boundary conditions weakly and naturally incorporate upwind stabilization \cite{liu2024multigrid,leykekhman2012local}. Nevertheless, DG methods may overlook extremely thin boundary layers, which is undesirable in applications where accurate resolution of the layer structure is required, even on relatively coarse meshes.

To address this numerical challenge, we adopt a monotone conforming finite element method, namely the edge-averaged finite element (EAFE) scheme \cite{xu1999monotone}.
This monotone discretization preserves the desired-state bounds in the discrete setting, thereby eliminating nonphysical oscillations and yielding stable, physically consistent numerical solutions.

Our contributions in this work are twofold.
First, we discretize the optimal control problem using the EAFE method and derive a concrete error estimate.
To achieve this, we combine the consistency results of \cite{xu1999monotone} with the discrete inf-sup condition established in \cite{brenner2020multigrid}.
The inf-sup condition not only guarantees the well-posedness of the discrete problem,
but also provides a direct framework for deriving {\textit{a priori}} error estimates.
This analysis can be regarded as an extension of the results in \cite{xu1999monotone} to optimal control problems.
We believe this is the first work to establish a concrete error estimate for the EAFE method applied to optimal control problems.
While our estimate does not explicitly track the dependence on the diffusion coefficient $\varepsilon$ and convection field $\bm{\zeta}$, which would require a more delicate analysis, we refer to \cite{liu2024multigrid} for results that include such parameter-dependent estimates.
Nevertheless, we show that, for sufficiently small $h$, the EAFE method achieves the optimal convergence order in the $H^1$ norm, consistent with the findings in \cite{xu1999monotone}.

Second, we prove that the monotone EAFE scheme preserves the desired-state bounds in the discrete setting, thereby guaranteeing the absence of spurious oscillations.
Several numerical studies have addressed convection-dominated optimal control problems using stabilization techniques such as streamline upwind/Petrov–Galerkin (SUPG)~\cite{collis2002analysis,heinkenschloss2010local}, AFC~\cite{baumgartner2025afc}, and DG formulations~\cite{leykekhman2012local,yucel2015discontinuous}.
While these approaches successfully demonstrated numerically that oscillations can be mitigated, a rigorous analytical explanation of this property has not yet been established.
To the best of our knowledge, this work provides the first mathematical proof that a monotone scheme preserves the stability relation (desired-state bounds) and completely eliminates nonphysical oscillations for optimal control problems constrained by convection-dominated equations.
This monotonicity property is ensured by the $M$-matrix condition of the stiffness matrix associated with the discrete convection-diffusion operator, which determines the direction of the inequalities in the discrete saddle-point system.

The remainder of the paper is organized as follows.
In Section~\ref{sec:prelim}, we introduce the notation, reformulate the optimal control problem as a saddle-point system, and discuss the monotonicity property together with the resulting desired-state bounds at the continuous level.
Section~\ref{sec: eafe_ocp} presents the EAFE method, the corresponding discrete formulation, and the derivation of the discrete desired-state bounds.
The error analysis of the EAFE scheme is carried out in Section~\ref{sec: error_analysis}.
In Section~\ref{sec: numerical_experiment}, we provide a series of numerical experiments that illustrate the performance of the proposed method for convection-dominated optimal control problems.
Finally, concluding remarks are given in Section~\ref{sec:conclusion}.


\section{Continuous Problems and Preliminaries} 
\label{sec:prelim}

In this section, we introduce the notation and basic definitions used throughout the paper, formulate the continuous optimal control problem, and recall key analytical properties such as the maximum principle and the associated desired-state bounds.
These results provide the theoretical foundation for the monotone discretization and analysis developed in the later sections.

Let $\mathcal{D} \subset \mathbb{R}^2$ be a bounded Lipschitz domain. For any real number $s \ge 0$, we denote by $H^s(\mathcal{D})$ the Sobolev space of order $s$, equipped with the norm norm $\|\cdot\|_{H^s(\mathcal{D})}$ and seminorm $|\cdot|_{H^s(\mathcal{D})}$. In particular, $H^0(\mathcal{D}) = L^2(\mathcal{D})$, and $(\cdot,\cdot)_{L^2(\mathcal{D})}$ denotes the $L^2$ inner product.
We denote by $H_0^1(\mathcal{D})$ the subspace of $H^1(\mathcal{D})$ consisting of functions with vanishing trace on $\partial\mathcal{D}$.

\subsection{Optimal Control Problems}
\label{subsec: OCPs}

We consider the following elliptic distributed optimal control problem. Let $\Omega \subset \mathbb{R}^2$ be a bounded convex polygonal domain, let $y_d \in L^2(\Omega)$ be a given desired state, and let $\beta > 0$ be a fixed regularization parameter. The goal is to find the optimal state-control pair $(\bar{y}, \bar{u})$ that minimizes the cost functional:
\begin{equation}\label{optcon}
\left(\bar{y},\bar{u}\right)=\argmin_{(y,u)}\left [ \frac{1}{2}\|y-y_d\|^2_{L^2(\Omega)}+\frac{\beta}{2}\|u\|^2_{L^2(\Omega)}\right],
\end{equation} 
subject to the constraint that $(y, u) \subset H_0^1(\Omega) \times L^2(\Omega)$ if and only if the state equation
\begin{equation}\label{eq:stateeq}
a(y,v)=(u,v)_{L^2(\Omega)}, \quad \forall v\in H^1_0(\Omega)
\end{equation}
is satisfied. The bilinear form $a(\cdot, \cdot)$ is defined by
\begin{equation}\label{eq:abilinear}
  a(y,v):=\int_\Omega \left(\varepsilon\nabla y+\bm{\zeta}y\right)\cdot \nabla v\,\dx+\int_{\Omega} \gamma yv\,\dx,
\end{equation}
where $\varepsilon(\bx) \geq \varepsilon_0 > 0$, the vector field $\bm{\zeta}(\bx) \in [W^{1, \infty}(\Omega)]^2$, and the reaction coefficient $\gamma(\bx) \in W^{1, \infty}(\Omega)$ is nonnegative.
We further assume that
\begin{equation}\label{eq:advassump}
    \gamma-\frac12\nabla\cdot\bm{\zeta}\ge\gamma_0>0,
\end{equation}
which guarantees the well-posedness of the state equation \eqref{eq:stateeq}.
In this work, we focus on the convection-dominated regime, characterized by $$\|\varepsilon\|_{L^\infty(\Omega)}\ll\|\bm{\zeta}\|_{L^\infty(\Omega)}.$$

Following \cite{Lions, Tro}, the solution of \eqref{optcon}--\eqref{eq:stateeq} is characterized by
\begin{subequations}\label{eq:sp}
\begin{alignat}{3}
a(q,\bar{p})&=(\bar{y}-y_d,q)_\LT, \quad &&\forall q\in H^1_0(\Omega),\\
\bar{p}+\beta\bar{u}&=0,\\
a(\bar{y},z)&=(\bar{u},z)_\LT,  \quad &&\forall z\in H^1_0(\Omega),
\end{alignat}
\end{subequations}
where $\bar{p}$ is the adjoint state.
After eliminating $\bar{u}$, we arrive at the saddle-point problem
\begin{subequations}\label{eq:osp}
\begin{alignat}{2}
a(q,\bar{p})-(\bar{y},q)_\LT&=-(y_d,q)_\LT, \quad &&\forall q\in H^1_0(\Omega),\\
-(\bar{p},z)_\LT-\beta a(\bar{y},z)&=0,  \quad &&\forall z \in H^1_0(\Omega).
\end{alignat}
\end{subequations}
By standard saddle-point theory \cite{brenner2020multigrid,Brezzi,Bab}, the system \eqref{eq:osp} has a unique solution. For simplicity, let $\beta=1$ and we consider the following saddle-point problem, 
\begin{subequations}\label{eq:speps}
\begin{alignat}{2}
a(q,\bar{p})-(\bar{y},q)_\LT&=-(y_d,q)_\LT, \quad &&\forall q\in H^1_0(\Omega),\\
-(\bar{p},z)_\LT-a(\bar{y},z)&=0,  \quad &&\forall z \in H^1_0(\Omega).
\end{alignat}
\end{subequations}
Note that, by integration by parts, the saddle-point problem \eqref{eq:speps} involves two convection-diffusion operators that are formally adjoint to each other, acting in opposite convection directions. Specifically, for all $(q,z)\in \Ho\times\Ho$,
\begin{subequations}\label{eq:speps1}
\begin{alignat*}{2}
\left[(\eps\nabla \bar{p}, \nabla q)_\LT-(\bm{\zeta}\bar{p}, \nabla q)_\LT+((\gamma-\nabla\cdot\bm{\zeta}) \bar{p}, q)_\LT\right]-(\bar{y},q)_\LT&=-(y_d,q)_\LT,\\[1ex]
-(\bar{p},z)_\LT-\left[(\eps\nabla \bar{y}, \nabla z)_\LT+(\bm{\zeta} \bar{y}, \nabla z)_\LT+(\gamma \bar{y}, z)_\LT\right]&=0.
\end{alignat*}
\end{subequations}
The saddle-point problem \eqref{eq:speps} can be written in the concise form
\begin{equation}\label{eq:conciseform}
\mathcal{B}((\bar{p},\bar{y}),(q,z))=-(y_d,q)_\LT,\quad \forall (q,z)\in H^1_0(\Omega)\times H^1_0(\Omega),
\end{equation}
where
\begin{equation}\label{eq:bilinear}
\mathcal{B}((p,y),(q,z)):=a(q,p)-(y,q)_\LT-(p,z)_\LT-a(y,z).
\end{equation}
Introducing the bilinear form $\mathcal{B}(\cdot,\cdot)$ allows the coupled state-adjoint system to be represented as a single variational problem, making the analysis of discrete inf-sup conditions and convergence more transparent and structurally unified.

\subsection{Maximum Principle and Stability via Desired-State Bounds}

We discuss a fundamental mathematical property of the optimal control problem constrained by the state equation. The state equation involves the elliptic operator
\[
\mathcal{L}y = -\nabla\cdot(\varepsilon\nabla y + \bm{\zeta}y) + \gamma y.
\]
Together with the boundary condition $y \in H_0^1(\Omega)$, the weak maximum principle ensures that the inverse of $\mathcal{L}$ is nonnegative. In other words, even in the convection-dominated regime, if $\mathcal{L}y(\mathbf{x}) \geq 0$ for all $\mathbf{x} \in \Omega$, then $y(\mathbf{x}) \geq 0$ for all $\mathbf{x} \in \Omega$. 
This property is often referred to as the monotonicity of the operator \cite[Theorem 2, Section 6.4]{evans10}. 

From the saddle-point system \eqref{eq:speps}, we obtain the following implications: For all $\mathbf{x}\in \Omega$,
\begin{align*}
\bar{y}(\mathbf{x}) - y_d(\mathbf{x}) \leq 0 \quad &\Rightarrow \quad \bar{p}(\mathbf{x}) \leq 0 \quad \Rightarrow \quad \bar{y}(\mathbf{x}) \geq 0,\\[0.5ex]
\bar{y}(\mathbf{x}) - y_d(\mathbf{x}) \geq 0 \quad &\Rightarrow \quad \bar{p}(\mathbf{x}) \geq 0 \quad \Rightarrow \quad \bar{y}(\mathbf{x}) \leq 0.
\end{align*}
Consequently, the optimal state $\bar{y}$ satisfies the following bounds determined by the desired state $y_d$:
\begin{equation}\label{eq:desired-state_bounds}
    \begin{cases}
        0 \leq \bar{y}(\mathbf{x}) \leq y_d(\mathbf{x}) & \text{for } y_d(\mathbf{x}) \geq 0,\  \forall\mathbf{x}\in\Omega,\\[0.7ex]
        y_d(\mathbf{x}) \leq \bar{y}(\mathbf{x}) \leq 0 & \text{for } y_d(\mathbf{x}) \leq 0, \ \forall\mathbf{x}\in\Omega.
    \end{cases}
\end{equation}
We call these inequalities as the \emph{desired-state bounds}.

\begin{remark}
It is important to emphasize that these bounds are not imposed as explicit constraints in the optimization problem. 
Rather, they arise naturally as a stability property of the optimal control problem, inherited from the saddle-point problem and the monotonicity of the elliptic operator. In other words, the output of the optimization problem (the optimal state $\bar{y}$) remains bounded by the given data of the problem (the desired state $y_d$), which directly reflects the concept of stability even in convection-dominated regimes.
\end{remark}

A monotone numerical scheme provides a discretization of $\mathcal{L}$ that preserves its monotonicity property; that is, if $\mathcal{L}_h y_h(\mathbf{x}) \geq 0$ for all $\mathbf{x} \in \Omega$, then $y_h(\mathbf{x}) \geq 0$ for all $\mathbf{x} \in \Omega$. Such schemes are well known to yield stable numerical solutions for convection-dominated problems without spurious oscillations (see \cite{xu1999monotone, cao2025edge} for further details). For optimal control problems constrained by convection-dominated state equations, we demonstrate that the monotone scheme preserves the desired-state bounds at the discrete level, preventing spurious oscillations (see Section~\ref{subsec: desired_state_bounds}).


\section{A Monotone Scheme for Optimal Control Problems}
\label{sec: eafe_ocp}

In this section, we develop and analyze a monotone discretization for optimal control problems constrained by convection-diffusion equations. Our approach builds on the edge-averaged finite element (EAFE) method, originally introduced as a monotone scheme for convection-diffusion problems, and extends it to the optimal control setting. We first formulate the discrete saddle-point system based on the EAFE bilinear form, and then show that the resulting scheme inherits the stability properties of the continuous problem. In particular, we establish that the discrete optimal state satisfies the discrete version of the desired-state bounds, which ensures robustness of the scheme and prevents spurious oscillations in convection-dominated regimes.

\subsection{Edge-Averaged Finite Element Discretization}

The EAFE discretization, originally introduced in \cite{xu1999monotone}, is a well-known monotone scheme for convection-diffusion problems. We present a flux-based formulation of the EAFE method that employs locally defined edge potentials together with mass lumping in the N\'ed\'elec space. This approach is particularly effective in convection-dominated regimes, where it provides stable and physically consistent numerical solutions.

For the discrete setting, let $\mathcal{T}_h$ be a shape-regular triangulation of $\Omega$, where each element $T \in \mathcal{T}_h$ is a triangle.
For any element $T \in \mathcal{T}_h$, we denote its diameter by $h_T$, and let $h = \max_{T \in \mathcal{T}_h} h_T$ denote the mesh size of the triangulation.

\subsubsection{Exponentially Fitted Flux and Its Approximation}

We first define the flux by $$J(y) := \varepsilon \nabla y+\bm{\zeta} y.$$
If there exists a scalar function $\psi$ such that $\nabla\psi  = \varepsilon^{-1}\bm{\zeta}$, then the flux can be rewritten as 
\begin{equation*}
J(y) = \kappa\nabla (e^{\psi} y),
\end{equation*}
where $\kappa(\bx):=\varepsilon(\bx) e^{-\psi(\bx)}$.
In general, when  $\textbf{curl}(\varepsilon^{-1}\bm{\zeta})\not=0$, such a global potential function $\psi$ may not exist. However, along any line segment $L$ between $\bx_0$ and $\bx_1$, we can define a one-dimensional edge potential $\psi_L$ through the relation
\begin{equation}\label{eq:edgepotential}
\psi_L(\bx_1) -\psi_L(\bx_0)= \frac{1}{|L|}\int_L \varepsilon^{-1}(\bm{\zeta}\cdot \bm{\tau}_L)\; {\rm d} s,
\end{equation}
where $\bm{\tau}_L=\bx_1-\bx_0$ is the scaled tangent vector to $L$ with $|\bm{\tau}_L|=|L|$.
With this definition, we set $\kappa_L:=\varepsilon e^{-\psi_L}$, and the flux along $L$ satisfies
$$
\left.J(y)\cdot \bm{\tau}_L \right|_L = \left.\kappa_L \nabla (e^{\psi_L} y)\cdot \bm{\tau}_L \right|_L.
$$ 

For the discrete formulation, let $\mathcal{T}_h$ be a conforming triangulation of the domain $\Omega$, and let $V_h \subset H_0^1(\Omega)$ denote the associated piecewise linear finite element space defined on $\mathcal{T}_h$. In the discrete setting, we consider the local flux
$$\left.J(y_h)\right|_T=\left.(\varepsilon \nabla y_h+\bm{\zeta} y_h)\right|_T$$ for $y_h\in V_h$ and each element $T\in\mathcal{T}_h$.
For each edge $E\subset\partial T$, let $\bm{\tau}_E$ denote the scaled tangent vector such that $|\bm{\tau}_E|=|E|$. 
We then define an edgewise potential function $\psi_E$ according to the relation \eqref{eq:edgepotential}.
With this definition, the flux along $E$ can be expressed as
\begin{equation}\label{eq:flux_express}
\left.J(y_h)\cdot\bm{\tau}_E\right|_E=\left.\kappa_E\nabla(e^{\psi_E}y_h)\cdot\bm{\tau}_E\right|_E.
\end{equation}

We denote by $\mathcal{N}_0(T)$ the lowest-order local N\'ed\'elec space, consisting of functions of the form 
\begin{equation*}
\mathbf{p}_0+\mathbf{x}^\perp q_0,\quad\mathbf{p}_0\in (\mathcal{P}_0(T))^2,\ q_0\in \mathcal{P}_0(T),    
\end{equation*}
where $\mathbf{x}^\perp$ denotes a $90^{\circ}$ rotation of $\mathbf{x}$ such that $\mathbf{x}\cdot\mathbf{x}^\perp=0$.
For any $\bv\in \mathcal{N}_0(T)$, the degrees of freedom are defined edgewise by
\begin{equation*}
\text{dof}_E(\mathbf{v})=\frac{1}{|E|}\int_E\mathbf{v}\cdot\bm{\tau}_E\;{\rm d}s=\left.(\mathbf{v}\cdot\bm{\tau}_E)\right|_E.
\end{equation*}
The canonical basis functions $\{\bm{\chi}_{E_j}\}_{j=1}^3\subset \mathcal{N}_0(T)$ are then characterized by
$$\text{dof}_{E_i}(\bm{\chi}_{E_j})=\delta_{ij},\quad 1\leq i,j\leq 3.$$

Applying $\text{dof}_E$ to the scaled flux $\kappa_E^{-1}J(y_h)|_T$ in \eqref{eq:flux_express}, we obtain
\begin{equation}
\text{dof}_{E}(\kappa_E^{-1}J(y_h))
=\delta_{E}(e^{\psi_E} y_h):=(e^{\psi_E} y_h)(\bx_j)-(e^{\psi_E} y_h)(\bx_i),
\label{def delta}
\end{equation}
where $\mathbf{x}_i$ and $\mathbf{x}_j$ are the endpoints of edge $E$, and $\bm{\tau}_E=\bx_j-\bx_i$.
Based on this relation, we define the flux approximation $J_T(y_h)\in \mathcal{N}_0(T)$ for $J(y_h)|_T$ by prescribing its degrees of freedom along each edge $E\subset\partial T$:
\begin{equation}
\text{dof}_E(J_T(y_h))=\left.(J_T(y_h)\cdot \bm{\tau}_E)\right|_E=H(\kappa_E)\delta_{E}(e^{\psi_E} y_h),
\label{edgeidentity}
\end{equation}
where $H(\kappa_E)$ denotes the harmonic average of $\kappa_E=\varepsilon e^{-\psi_E}$ over the edge $E$, defined as
\begin{equation}
H(\kappa_E)=\left(\frac{1}{|E|}\int_{E}\kappa_E^{-1}\;{\rm d}s\right)^{-1}.\label{harmonic average}
\end{equation}

Thus, the flux approximation $J_T(y_h)\in \mathcal{N}_0(T)$ can be written in terms of the canonical basis functions as
\begin{equation*}
J_T (y_h)= \sum_{E\subset \partial T}\text{dof}_E(J_T(y_h))\bm{\chi}_E=\sum_{E\subset \partial T}H({\kappa}_{E})\delta_{E}(e^{\psi_E} y_h)\bm{\chi}_E.
\end{equation*}
Moreover, since $\left.\nabla v_h\right|_T\in \mathcal N_0(T)$, it admits a similar representation:
$$
\nabla v_h=\sum_{E\subset \partial T} (\nabla v_h\cdot\bm{\tau}_E) \bm{\chi}_E = \sum_{E\subset \partial T}\delta_{E}(v_h)\bm{\chi}_E.
$$

\subsubsection{Edge-Averaged Finite Element (EAFE) Scheme}\label{EAFE}

Let $\{\lambda_i(\mathbf{x})\}_{i=1}^3\subset \mathcal{P}_1(T)$ denote the barycentric coordinates associated with the vertices $\{\mathbf{x}_i\}_{i=1}^3$ of an element $T\in \mathcal{T}_h$.
We begin by introducing a mass-lumping approximation for the local mass matrix in the N\'ed\'elec space $\mathcal{N}_0(T)$:
\begin{equation}
\int_T\mathbf{v}\cdot\mathbf{w}\;{\rm d}\mathbf{x}\approx\sum_{E\subset \partial T}\omega_E^T\left.(\mathbf{v}\cdot\bm{\tau}_E)\right|_E\left.(\mathbf{w}\cdot\bm{\tau}_E)\right|_E,\quad\forall\bv,\bw\in \mathcal{N}_0(T),\label{eqn: mass lumping approx}
\end{equation}
where the edge weights $\omega_E^T$ are defined by
\begin{equation*}
\omega_E^T=-\int_T\nabla\lambda_i\cdot\nabla\lambda_j\;{\rm d}\mathbf{x}\quad\text{for the edge $E$ with }\bm{\tau}_E=\mathbf{x}_j-\mathbf{x}_i.
\end{equation*}
Since both $J_T(y_h)$ and $\nabla v_h|_T$ belong to $\mathcal{N}_0(T)$, we apply the mass-lumping approximation to their inner product:
\begin{equation*}
\int_T J_T(y_h)\cdot\nabla v_h\;{\rm d}\mathbf{x}\approx
\sum_{E\subset \partial T}\omega_E^T H({\kappa}_E)\delta_E(e^{\psi_E} y_h)\delta_E(v_h).
\end{equation*}
This leads to the definition of the local EAFE bilinear form:
\begin{equation*}
    a_h^T(y_h,v_h) := \sum_{E\subset \partial T}\omega_E^TH(\kappa_E)\delta_E(e^{\psi_E} y_h)\delta_E(v_h),
\end{equation*}
where $H(\kappa_E)$ is the harmonic average of $\kappa_E$ on edge $E$, as defined in \eqref{harmonic average}. The corresponding global EAFE bilinear form is then given by
\begin{equation}\label{eq:eafebi}
  a_h(y_h,v_h)=\sum_{T\in\Th}a_h^T(y_h,v_h).
\end{equation}

\begin{lemma}[Consistency estimate] Assume that $\varepsilon\in W^{1,\infty}(T)$ and $\bm{\zeta}\in [W^{1,\infty}(T)]^2$ for all $T\in\mathcal{T}_h$. Then, there exists a constant $C>0$, independent of $h$, such that
\begin{equation}
    \left|a(y_h,v_h) - a_h(y_h,v_h)\right|\leq Ch\norm{y_h}_{H^1(\Omega)}\norm{v_h}_{H^1(\Omega)}.\label{eq:estimate_EAFE}
\end{equation}
A proof of this estimate can be found in \cite{xu1999monotone}.
\end{lemma}

\begin{remark}
    Regarding the bilinear forms, note that $\bar{\kappa}_E>0$ and that $$\delta_E(e^{\psi_E} \lambda_i)\delta_E(\lambda_j)=-e^{\psi_E(\mathbf{x}_i)}<0.$$ 
    This implies that the off-diagonal entries of the stiffness matrix associated with the global bilinear form are nonpositive if and only if $$\omega_E^T+\omega_E^{T'}\geq 0$$ for every interior edge $E=\partial T\cap \partial T'$.
    This condition is equivalent to the Delaunay condition on the triangulation $\mathcal{T}_h$ in two dimensions (see \cite{xu1999monotone} for further details).
    Therefore, the stiffness matrix corresponding to the EAFE bilinear form is an $M$-matrix if and only if the stiffness matrix for the Poisson problem (with weights $-\omega_E^T$) is also an $M$-matrix. In particular, any invertible $M$-matrix has a positive inverse.
\end{remark}

\begin{remark}
For implementation purposes, the coefficients $\varepsilon$ and $\bm{\zeta}$ are approximated along each edge $E\subset\partial T$ by their midpoint averages, denoted by $\varepsilon_E$ and $\bm{\zeta}_E$, respectively.
For example, $$\varepsilon_E=\frac{\varepsilon(\mathbf{x}_i)+\varepsilon(\mathbf{x}_j)}{2}\quad \text{and}\quad \bm{\zeta}_E=\frac{\bm{\zeta}(\mathbf{x}_i)+\bm{\zeta}(\mathbf{x}_j)}{2},$$
where $\bx_i$ and $\bx_j$ are the endpoints of edge $E$.
The corresponding edgewise potential function is then defined as $$\psi_E(\mathbf{x})={\varepsilon_E}^{-1}\bm{\zeta}_E\cdot\mathbf{x},$$
which is uniquely determined up to an additive constant.
With this approximation, the degrees of freedom in \eqref{edgeidentity} can be explicitly evaluated as
\begin{equation}\label{eqn: easy_computing}
H(\kappa_E)\delta_E(e^{\psi_E} y_h)=\varepsilon_E \mathbb{B}({\varepsilon_E}^{-1}\bm{\zeta}_E\cdot(\mathbf{x}_{i}-\mathbf{x}_{j}))y_h(\mathbf{x}_j)-\varepsilon_E \mathbb{B}({\varepsilon_E}^{-1}\bm{\zeta}_E\cdot(\mathbf{x}_{j}-\mathbf{x}_{i}))y_h(\bx_{i}),
\end{equation}
where $\mathbb{B}(x)$ denotes the Bernoulli function,
\begin{equation}
\mathbb{B}(x)=
\left\{
\begin{array}{cl}
\displaystyle\frac{x}{e^x-1}&  x\not=0,\\[2ex]
1& x=0.
\end{array}
\right.
\label{bernoulli function}
\end{equation}
\end{remark}

\subsection{A Monotone Scheme for Optimal Control Problems}

We now present the monotone discretization for the optimal control problem. The discrete problem is to find $(\bar{p}_h,\bar{y}_h)\in V_h\times V_h$ such that
\begin{equation}\label{eq:disprob_with_yd}
\mathcal{B}_h((\bar{p}_h,\bar{y}_h),(q_h,z_h))=-(y_d,q_h)_{L^2(\Omega)},\quad \forall (q_h,z_h)\in V_h\times V_h,
\end{equation}
where the bilinear form $\cB(\cdot,\cdot)$ is defined as 
\begin{equation}\label{eq:bilinearh}
\mathcal{B}_h((p,y),(q,z))= a_h(q,p) -(y,q)_\LT-(p,z)_\LT-a_h(y,z),
\end{equation}
with $a_h(\cdot,\cdot)$ defined in \eqref{eq:eafebi}.
Equivalently, this can be expressed as the following saddle-point system:
\begin{subequations}\label{eq:discrete_speps}
\begin{alignat}{2}
a_h(q_h,\bar{p}_h)-(\bar{y}_h,q_h)_\LT&=-(y_d,q_h)_\LT, \quad &&\forall q_h\in V_h,\\
-(\bar{p}_h,z_h)_\LT-a_h(\bar{y}_h,z_h)&=0,  \quad &&\forall z_h \in V_h.
\end{alignat}
\end{subequations}

\begin{remark}
Note that $a_h(q,p)$ may not be the EAFE discretization of the bilinear form $a(q,p)$. However, we define our discrete problem using $a_h(q,p)$ so that we can utilize the dual property in the analysis. 
\end{remark}

\begin{remark}
    We could also consider a general saddle-point formulation: find $(p,y)\in\Ho\times\Ho$ such that
\begin{equation}\label{eq:general}
\mathcal{B}((p,y),(q,z))=(f,q)+(g,z), \quad \forall (q,z)\in \Ho\times \Ho,
\end{equation}
where $(f,g)\in\LT\times\LT$ and the bilinear form $\cB(\cdot,\cdot)$ is defined in \eqref{eq:bilinear}. The EAFE discretization of \eqref{eq:general} is to find $(p_h,y_h)\in V_h\times V_h$ such that
\begin{equation}\label{eq:disprob}
\mathcal{B}_h((p_h,y_h),(q_h,z_h))=(f,q_h)+(g,z_h), \quad \forall (q_h,z_h)\in V_h\times V_h.
\end{equation}
This general formulation is more practical if an external force is included in the state equation \eqref{eq:stateeq}, it is also useful when designing numerical experiments.
\end{remark}

\subsection{Preservation of Desired-State Bounds}\label{subsec: desired_state_bounds}

We now establish a discrete version of the desired-state bounds \eqref{eq:desired-state_bounds}, demonstrating that they are preserved by a monotone scheme for optimal control problems subject to convection-dominated state equations.
\begin{theorem} \label{thm: desired-state_bounds}
Let $\bar{y}_h$ be the discrete optimal state of the saddle-point problem \eqref{eq:discrete_speps}, and let $y_d$ denote the desired state.
For each vertex $\mathbf{x}_i$ of the mesh $\mathcal{T}_h$, let $\phi_i \in V_h$ be the corresponding canonical nodal basis function, where $1 \le i \le N_V$ and $N_V$ is the number of vertices in $\mathcal{T}_h$.
Then, for every $1\leq i\leq N_V$, the discrete optimal state satisfies the inequalities
  \begin{equation*}
    \begin{cases}
        0 \leq (\bar{y}_h,\phi_i )_\LT\leq (y_d,\phi_i)_\LT & \text{for } y_d(\mathbf{x}) \geq 0,\  \forall\mathbf{x}\in\Omega, \\[0.5ex]
        (y_d,\phi_i)_\LT \leq (\bar{y}_h,\phi_i)_\LT \leq 0 & \text{for } y_d(\mathbf{x}) \leq 0,\ \forall\mathbf{x}\in\Omega.
    \end{cases}
\end{equation*}
Additionally, the discrete adjoint state $\bar{p}_h$ satisfies the inequalities
\begin{equation*}
    \begin{cases}
        \bar{p}_h\leq0 & \text{for } y_d(\mathbf{x}) \geq 0,\ \forall\mathbf{x}\in \Omega, \\[0.5ex]
        \bar{p}_h\geq0 & \text{for } y_d(\mathbf{x}) \leq 0,\ \forall\mathbf{x}\in \Omega.
    \end{cases}
\end{equation*}
\end{theorem}

\begin{proof}
Define the coefficient vectors corresponding to the discrete optimal state $\bar{y}_h$ and the adjoint state $\bar{p}_h$:
\begin{equation*}
    \bar{\mathsf{y}} = \left[ \bar{y}_h(\mathbf{x}_1),\ \bar{y}_h(\mathbf{x}_2),\ \cdots,\  \bar{y}_h(\mathbf{x}_{N_V}) \right]^\top,\quad \bar{\mathsf{p}} = \left[ \bar{p}_h(\mathbf{x}_1),\ \bar{p}_h(\mathbf{x}_2),\ \cdots,\  \bar{p}_h(\mathbf{x}_{N_V}) \right]^\top.
\end{equation*}
The discrete saddle-point problem \eqref{eq:discrete_speps} with the EAFE discretization leads to the linear system
\begin{subequations}\label{eq:system_speps}
\begin{alignat}{2}
\mathsf{A}^\top\bar{\mathsf{p}}-\mathsf{M}\bar{\mathsf{y}}&=-\mathsf{f}(\mathsf{y}_\mathsf{d}),\label{eq:system_speps_1}\\
-\mathsf{M}\bar{\mathsf{p}} - \mathsf{A}\bar{\mathsf{y}}&=\mathsf{0},\label{eq:system_speps_2}
\end{alignat}
\end{subequations}
where $\mathsf{A}_{ij} = a_h(\phi_j,\phi_i)$ is the stiffness matrix satisfying the $M$-matrix condition, $\mathsf{M}_{ij}=(\phi_j,\phi_i)_\LT$ is the mass matrix with nonnegative entries, and $(\mathsf{f}(\mathsf{y}_d))_i = (y_d,\phi_i)_\LT$ is the right-hand side vector for $1\leq i,j\leq N_V$.

Consider first the case $y_d \leq 0$ and assume
\begin{equation}\label{eqn: bound_int_1}
    (\bar{y}_h-y_d,\phi_i)_\LT\geq0\quad\text{for all }1\leq i\leq N_V.
\end{equation}
Then, $\mathsf{M}\bar{\mathsf{y}}-\mathsf{f}(\mathsf{y}_\mathsf{d})\geq\mathsf{0}$, i.e., each component of this vector is nonnegative.
Since $\mathsf{A}$ is an invertible $M$-matrix with a nonnegative inverse,
\begin{equation*}
    \bar{\mathsf{p}}\geq (\mathsf{A}^\top)^{-1}(\mathsf{M}\bar{\mathsf{y}}-\mathsf{f}(\mathsf{y}_\mathsf{d}))\geq\mathsf{0}.
\end{equation*}
From \eqref{eq:system_speps_2}, we obtain $\mathsf{A}\bar{\mathsf{y}} = -\mathsf{M}\bar{\mathsf{p}}$, and the $M$-matrix property implies
\begin{equation*}
    \bar{\mathsf{y}} = \mathsf{A}^{-1}\mathsf{M}(-\bar{\mathsf{p}})\leq \mathsf{0},
\end{equation*}
hence $\mathsf{M}\bar{\mathsf{y}} \leq \mathsf{0}$. Combining this with assumption \eqref{eqn: bound_int_1}, we deduce
\begin{equation*}
     (y_d,\phi_i)_\LT\leq (\bar{y}_h,\phi_i)_\LT \leq 0.
\end{equation*}

In the opposite case $y_d \geq 0$, if
\begin{equation*}
    (\bar{y}_h-y_d,\phi_i)_\LT\leq0\quad\text{for all }1\leq i\leq N_V,
\end{equation*}
then an identical argument yields
\begin{equation*}
     (y_d,\phi_i)_\LT\geq (\bar{y}_h,\phi_i)_\LT \geq 0.
\end{equation*} 
\end{proof}

\begin{remark}
    This result in Theorem~\ref{thm: desired-state_bounds} is consistent with the desired-state bounds satisfied by the continuous optimal state $\bar{y}$. 
Since each canonical basis function $\phi_i$ is nonnegative, the $L^2$ inner products $(\bar{y}_h,\phi_i)_{L^2(\Omega)}$ reveal that the discrete optimal state $\bar{y}_h$ is bounded strongly on one side by zero and weakly on the other side by the desired state $y_d$. 
Thus, the preservation of desired-state bounds in the discrete setting ensures that the discrete optimal state $\bar{y}_h$ remains free from spurious oscillations, as such oscillations would otherwise produce overshoots or undershoots beyond the desired-state bounds.
\end{remark}

\section{Error Analysis}
\label{sec: error_analysis}

In this section, we analyze the well-posedness and convergence of the EAFE discretization \eqref{eq:disprob_with_yd} for the optimal control problem \eqref{eq:speps}.
We first establish the discrete inf-sup condition for the bilinear form $\mathcal{B}_h(\cdot,\cdot)$, which guarantees the stability and uniqueness of the discrete solution.
We then derive convergence results based on the consistency of the EAFE bilinear form $a_h(\cdot,\cdot)$ in \eqref{eq:estimate_EAFE} and the regularity of the continuous solution.

\subsection{Well-Posedness of the Discrete Problem}

We now establish a discrete inf-sup condition for \eqref{eq:disprob_with_yd}, ensuring its well-posedness via standard saddle-point theory (see \cite{Brezzi,Bab}). We follow the approach of \cite[Section 6]{xu1999monotone} and \cite{brenner2020multigrid}.
\begin{lemma}\label{lemma:bhinfsup}
There exists a constant $C>0$ such that for all $(p_h,y_h)\in V_h\times V_h$ and for sufficiently small $h$,
\begin{equation}\label{eqn: infsup_EAFE}
\sup_{(q_h,z_h)\in V_h\times V_h}\frac{\mathcal{B}_h((p_h,y_h),(q_h,z_h))}{\|q_h\|_{H^1(\Omega)}+\|z_h\|_{H^1(\Omega)}}\ge C\left(\|p_h\|_{H^1(\Omega)}+\|y_h\|_{H^1(\Omega)}\right).
\end{equation}
\end{lemma}
\begin{proof}
  Given $(p_h,y_h)\in V_h\times V_h$, choose the test functions $q_h=p_h-y_h$ and $z_h=-p_h-y_h$ in the bilinear form \eqref{eq:bilinearh}. Then, we obtain
  \begin{equation}
    \mathcal{B}_h((p_h,y_h),(p_h-y_h,-p_h-y_h))=a_h(p_h,p_h)+(y_h,y_h)_\LT+(p_h,p_h)_\LT+a_h(y_h,y_h).
  \end{equation}

  We estimate only the term $a_h(p_h,p_h)$, since the analysis for $a_h(y_h,y_h)$ is identical. Following \cite[Lemma 6.2]{xu1999monotone}, we decompose
  \begin{equation}
    a_h(p_h,p_h)=a(p_h,p_h)+\left[a_h(p_h,p_h)-a(p_h,p_h)\right].
  \end{equation}
  By integration by parts and the assumption \eqref{eq:advassump}, we obtain
  \begin{equation}\label{eq:app}
    a(p_h,p_h)=\|\eps\nabla p_h\|^2_\LT+\int_\Omega \left(\gamma-\frac12\nabla\cdot\bm{\zeta}\right) p_h^2\,\dx\ge \eps_0\|\nabla p_h\|^2_\LT+\gamma_0\|p_h\|^2_\LT.
  \end{equation}
  Furthermore, by the consistency estimate \eqref{eq:estimate_EAFE}, it follows that
  \begin{equation}\label{eq:ahppdiff}
    |a(p_h,p_h)-a_h(p_h,p_h)|\le Ch\|p_h\|^2_{H^1(\Omega)},
  \end{equation}
  where the constant $C>0$ is independent of $h$.
  For sufficiently small $h$, combining \eqref{eq:app} and \eqref{eq:ahppdiff} yields
  \begin{equation}\label{eq:ahpp}
    a_h(p_h,p_h)=a(p_h,p_h)+[a_h(p_h,p_h)-a(p_h,p_h)]\ge C\|p_h\|^2_{H^1(\Omega)}.
  \end{equation}
  
  Consequently, we obtain
  \begin{equation}\label{eq:bhcoer}
    \mathcal{B}_h((p_h,y_h),(p_h-y_h,-p_h-y_h))\ge C\left(\|p_h\|^2_{H^1(\Omega)}+\|y_h\|^2_{H^1(\Omega)}\right).
  \end{equation}
Using the parallelogram identity,
  \begin{equation}\label{eq:parallellaw}
    \|p_h-y_h\|^2_{H^1(\Omega)}+\|p_h+y_h\|^2_{H^1(\Omega)}=2\left((\|p_h\|^2_{H^1(\Omega)}+\|y_h\|^2_{H^1(\Omega)}\right),
  \end{equation}
  we conclude that the discrete inf-sup condition \eqref{eqn: infsup_EAFE} follows immediately from \eqref{eq:bhcoer} and \eqref{eq:parallellaw}.
\end{proof}

\subsection{Convergence Results}

Based on the discrete inf-sup condition established in Lemma~\ref{lemma:bhinfsup}, we derive the following error estimates for the EAFE discretization of the optimal control problem. 

Let $I_h: H^1(\Omega)\rightarrow V_h$ be the standard nodal interpolation operator. Then, the standard approximation estimates hold (cf.~\cite{BS}):
\begin{subequations}\label{eq:interpolation_estimates}
\begin{align}
    \|\bar{y}-I_h\bar{y}\|_{H^1(\Omega)} + \|\bar{p}-I_h\bar{p}\|_{H^1(\Omega)}
    &\leq C h \big( \|\bar{y}\|_{H^2(\Omega)} + \|\bar{p}\|_{H^2(\Omega)} \big),\\
    \|\bar{y}-I_h\bar{y}\|_{L^2(\Omega)} + \|\bar{p}-I_h\bar{p}\|_{L^2(\Omega)}
    &\leq C h^2 \big( \|\bar{y}\|_{H^2(\Omega)} + \|\bar{p}\|_{H^2(\Omega)} \big),
    \end{align}
\end{subequations}
where the constant $C>0$ is independent of $h$.
In addition, the bilinear form $a(\cdot,\cdot)$ satisfies the continuity condition
\begin{equation}\label{eq:continuity_a}
    |a(y,v)| \leq C \norm{y}_{H^1(\Omega)}\norm{v}_{H^1(\Omega)},
\end{equation}
with constant $C>0$ independent of $h$.

\begin{theorem}\label{thm:main}
  Let $(\bar{p},\bar{y})$ be the solution of \eqref{eq:conciseform} and $(\bar{p}_h,\bar{y}_h)$ be the solution of \eqref{eq:disprob_with_yd}.  
Assume that $\bar{p},\bar{y}\in H^2(\Omega)$. Then, for $h$ sufficiently small, the following error bound holds:
  \begin{equation}\label{eq:erroresti}
    \|\bar{y}-\bar{y}_h\|_{H^1(\Omega)}+\|\bar{p}-\bar{p}_h\|_{H^1(\Omega)}\le C h.
  \end{equation}
\end{theorem}

\begin{proof}
  By choosing $p_h = I_h\bar{p} - \bar{p}_h$ and $y_h = I_h\bar{y} - \bar{y}_h$ in the discrete inf-sup condition \eqref{eqn: infsup_EAFE} of Lemma~\ref{lemma:bhinfsup}, there exists a constant $C>0$ such that
  \begin{equation*}
    \|I_h\bar{p}-\bar{p}_h\|_{H^1(\Omega)}+\|I_h\bar{y}-\bar{y}_h\|_{H^1(\Omega)}\le C \sup_{(q_h,z_h)\in V_h\times V_h}\frac{\mathcal{B}_h((I_h\bar{p}-\bar{p}_h,I_h\bar{y}-\bar{y}_h),(q_h,z_h))}{\|q_h\|_{H^1(\Omega)}+\|z_h\|_{H^1(\Omega)}}.
  \end{equation*}
    From the definition of the discrete bilinear form $\mathcal{B}_h(\cdot,\cdot)$, together with \eqref{eq:disprob_with_yd} and \eqref{eq:conciseform}, we obtain
\begin{align*}
    \mathcal{B}_h((I_h\bar{p}-\bar{p}_h,I_h\bar{y}-\bar{y}_h),(q_h,z_h)) &= \mathcal{B}_h((I_h\bar{p},I_h\bar{y}),(q_h,z_h)) - \mathcal{B}_h((\bar{p}_h,\bar{y}_h),(q_h,z_h))\\
    &= \mathcal{B}_h((I_h\bar{p},I_h\bar{y}),(q_h,z_h)) +(y_d,q_h)_{L^2(\Omega)}\\
    &= \mathcal{B}_h((I_h\bar{p},I_h\bar{y}),(q_h,z_h)) - \mathcal{B}((\bar{p},\bar{y}),(q_h,z_h)).
\end{align*}
Expanding $\mathcal{B}_h(\cdot,\cdot)$ and $\mathcal{B}(\cdot,\cdot)$, we decompose the difference into four contributions:
\begin{align*}
    &\mathcal{B}_h((I_h\bar{p},I_h\bar{y}),(q_h,z_h)) - \mathcal{B}((\bar{p},\bar{y}),(q_h,z_h))\\[1ex]
    &\qquad\qquad\qquad=\  \underbrace{a_h(q_h,I_h\bar{p})-a(q_h,\bar{p})}_{(\textrm{I})}\ +\  \underbrace{( \bar{y}-I_h\bar{y},q_h)_{L^2(\Omega)}}_{\textrm{(II)}}\\
    &\qquad\qquad\qquad\quad\quad +\  \underbrace{(\bar{p}-I_h\bar{p},z_h)_{L^2(\Omega)}}_{\textrm{(III)}} \ +\  \underbrace{a(\bar{y},z_h)- a_h(I_h\bar{y},z_h)}_{\textrm{(IV)}}.
\end{align*}

\noindent\textbf{Estimate of $\textnormal{(I)}$:}  Applying the triangle inequality, the EAFE consistency estimate \eqref{eq:estimate_EAFE}, the continuity of $a(\cdot,\cdot)$ \eqref{eq:continuity_a}, and the interpolation error bounds \eqref{eq:interpolation_estimates} (neglecting $h^2$-terms for sufficiently small $h$), we obtain
\begin{align*}
    \textrm{(I)} &\leq |a_h(q_h,I_h\bar{p}) - a(q_h,I_h\bar{p})| + |a(q_h,I_h\bar{p}) - a(q_h,\bar{p})|\\
    &\leq Ch\norm{q_h}_{H^1(\Omega)}\norm{I_h\bar{p}}_{H^1(\Omega)} + C\norm{q_h}_{H^1(\Omega)}\norm{I_h\bar{p}-\bar{p}}_{H^1(\Omega)}\\
    &\leq Ch\norm{q_h}_{H^1(\Omega)}\norm{\bar{p}}_{H^2(\Omega)}.
\end{align*}

\noindent\textbf{Estimate of $\textnormal{(II)}$:} By the Cauchy-Schwarz inequality and the interpolation estimate \eqref{eq:interpolation_estimates},
\begin{align*}
    \textrm{(II)}\leq \norm{\bar{y}-I_n\bar{y}}_{L^2(\Omega)}\norm{q_h}_{L^2(\Omega)}\leq Ch^2\norm{\bar{y}}_{H^2(\Omega)}\norm{q_h}_{H^1(\Omega)}.
\end{align*}

\noindent\textbf{Estimate of $\textnormal{(III)}$:} Similarly,
\begin{equation*}
    \textrm{(III)}\leq Ch^2\norm{\bar{p}}_{H^2(\Omega)}\norm{z_h}_{H^1(\Omega)}.
\end{equation*}

\noindent\textbf{Estimate of $\textnormal{(IV)}$:} Proceeding as in (I),
\begin{align*}
    \textrm{(IV)}&\leq |a(\bar{y},z_h) - a(I_h\bar{y},z_h)| + |a(I_h\bar{y},z_h) - a_h(I_h\bar{y},z_h)|\\
    & \leq C\norm{{\bar{y} - I_h\bar{y}}}_{H^1(\Omega)}\norm{z_h}_{H^1(\Omega)} + Ch\norm{{I_h\bar{y}}}_{H^1(\Omega)}\norm{z_h}_{H^1(\Omega)}\\
    &\leq Ch\norm{\bar{y}}_{H^2(\Omega)}\norm{z_h}_{H^1(\Omega)}.
\end{align*}
Collecting the above estimates, we obtain
\begin{align*}
    \mathcal{B}_h((I_h\bar{p}-\bar{p}_h,I_h\bar{y}-\bar{y}_h),(q_h,z_h))\leq Ch\left(\norm{\bar{p}}_{H^2(\Omega)}+\norm{\bar{y}}_{H^2(\Omega)}\right)\left(\norm{q_h}_{H^1(\Omega)}+\norm{z_h}_{H^1(\Omega)}\right).
\end{align*}
Substituting into the discrete inf-sup condition yields the desired bound:
\begin{align*}
    \norm{I_h\bar{p} - \bar{p}_h}_{H^1(\Omega)} + \norm{I_h\bar{y} - \bar{y}_h}_{H^1(\Omega)}\leq Ch\left(\norm{\bar{p}}_{H^2(\Omega)}+\norm{\bar{y}}_{H^2(\Omega)}\right).
\end{align*}
\end{proof}

\begin{remark}
    The estimate \eqref{eq:erroresti} shows that, for sufficiently small $h$, the EAFE discretization \eqref{eq:disprob_with_yd} converges with first-order accuracy, i.e., at a rate of $O(h)$. This result does not account for the influence of the coefficients $\eps$ and $\bm{\zeta}$, whose effects are considerably more subtle to analyze. For a more detailed convergence analysis incorporating these effects in the context of an upwind DG scheme, we refer to \cite{liu2024multigrid}.
\end{remark}


\section{Numerical Experiments}
\label{sec: numerical_experiment}

In this section, we present a series of numerical experiments to verify the theoretical results obtained in the previous sections and to illustrate the performance of the proposed EAFE method~\eqref{eq:disprob_with_yd}. 
All computations are performed in \textsc{MATLAB} using the $i$FEM package~\cite{chen2008ifem}.

Two types of numerical tests are considered. 
In the \emph{stability test} (Section~\ref{subsec: stability_test}), we examine the optimal control formulation defined in problems~\eqref{eq:speps} and~\eqref{eq:discrete_speps}.
Here, $(\bar{y},\bar{p})$ and $(\bar{y}_h,\bar{p}_h)$ denote the continuous and discrete optimal state-adjoint pairs, respectively. 
This test investigates whether the monotone EAFE scheme preserves the desired-state bounds and prevents spurious oscillations in convection-dominated regimes.

In contrast, the \emph{convergence tests} (Sections~\ref{subsec: convergence_tests_boundary} and~\ref{subsec: convergence_tests_interior}) are based on the general coupled problem~\eqref{eq:general} and its discrete counterpart~\eqref{eq:disprob}. 
We denote the numerical errors by $e_y = y - y_h$ and $e_p = p - p_h$, where $(y,p)$ and $(y_h,p_h)$ are the continuous and discrete solutions, respectively. 
These tests include both boundary layer and interior layer examples and are designed to assess the accuracy and robustness of the method. 
For each example, we evaluate the errors both globally over the entire domain~$\Omega$ and locally in interior subregions away from the boundary or interior layers.

\subsection{Stability Test (Desired-State Bounds)}
\label{subsec: stability_test}

We perform a numerical experiment to verify the preservation of the desired-state bounds established in Theorem~\ref{thm: desired-state_bounds}. To clearly illustrate this property, we consider a simple desired state $y_d=1$, which allows for straightforward comparison with the theoretical prediction.

\begin{example}[Desired-State Bound]\label{ex:dsbound}

We consider the problem \eqref{eq:speps} on the domain $\Omega=(0,1)^2$, with parameters $\gamma=0$, $\bm{\zeta}=[-1,0]^T$, $\eps=10^{-9}$, and the desired state $y_d=1$.
\end{example}

\begin{figure}[h!]
    \centering
    \begin{tabular}{ccc}
    Optimal state $\bar{y}_h$ (monotone method) & & Optimal state $\bar{y}_h$ (standard method)\\
    \\
      \includegraphics[width=0.37\linewidth]{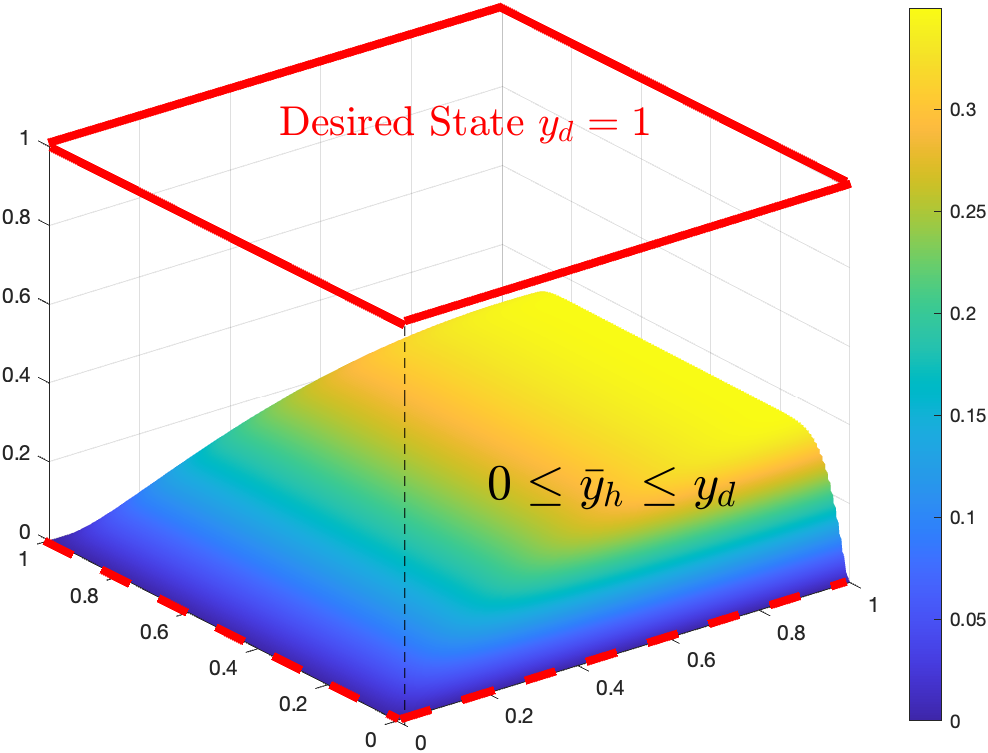} &   & \includegraphics[width=0.37\linewidth]{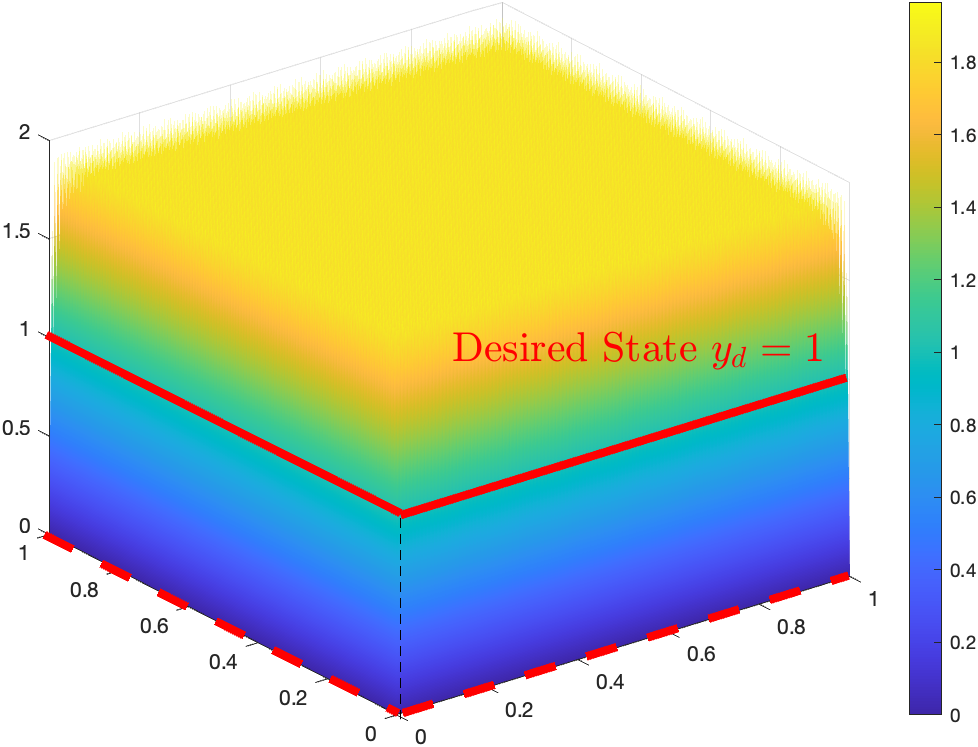} \\
      \\
      Adjoint state $\bar{p}_h$ (monotone method) & & Adjoint state $\bar{p}_h$ (standard method)\\
    \\
      \includegraphics[width=0.37\linewidth]{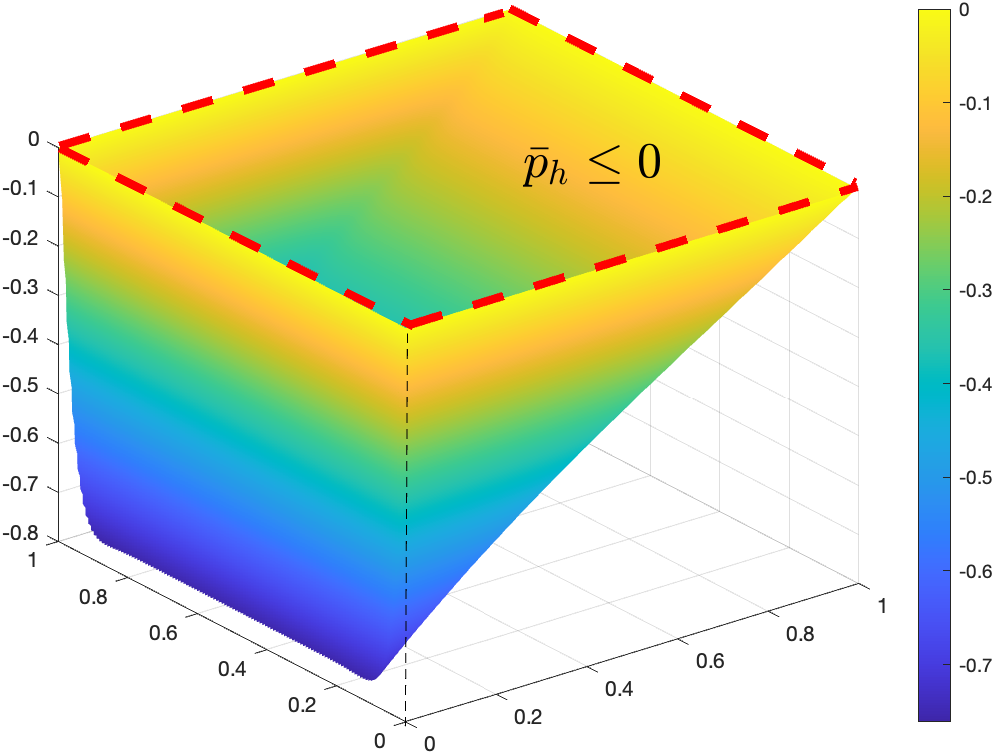} &   & \includegraphics[width=0.37\linewidth]{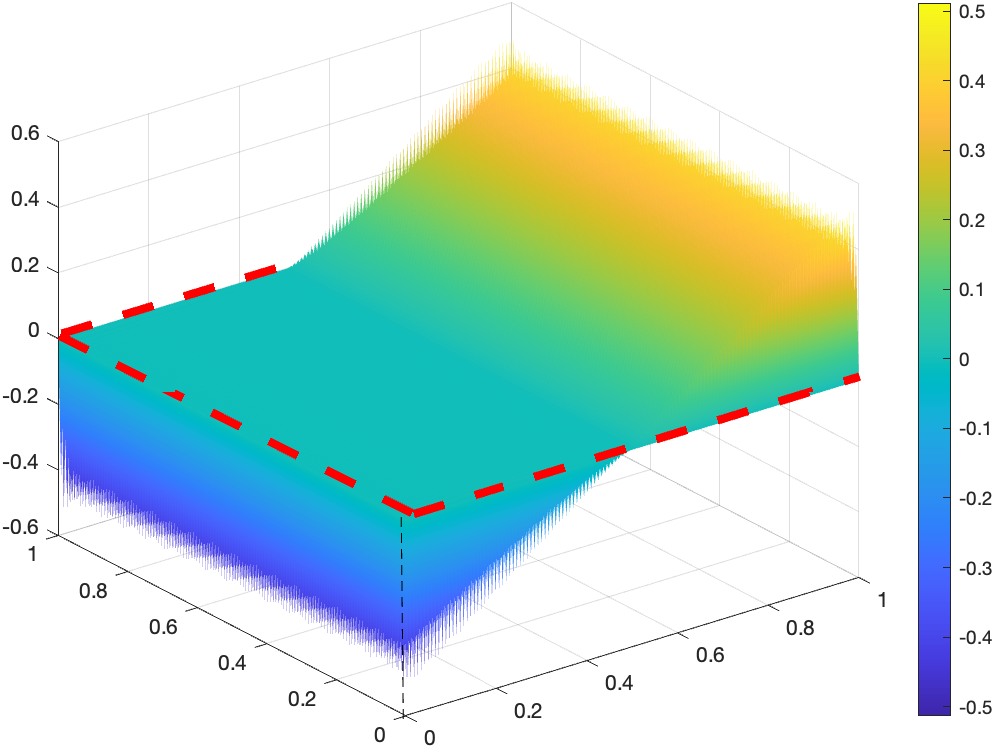}
    \end{tabular}
    \caption{Optimal and adjoint states for Example~\ref{ex:dsbound} with $\varepsilon = 10^{-9}$.}\label{fig:statebounds}
\end{figure}

Figure~\ref{fig:statebounds} demonstrates that the monotone numerical scheme, the EAFE method, produces an optimal state $\bar{y}_h$ that remains strictly between zero and the desired state $y_d$, in full agreement with the theoretical bounds given in Theorem~\ref{thm: desired-state_bounds}. In contrast, the standard (unstabilized) finite element method exhibits nonphysical oscillations, causing the computed optimal state to violate the desired-state bounds. Furthermore, consistent with Theorem~\ref{thm: desired-state_bounds}, the monotone scheme also preserves the bounds for the discrete adjoint state, whereas the standard approach fails to do so. These results confirm that the monotone method retains the key analytical stability properties of elliptic distributed optimal control problems, yielding physically consistent and oscillation-free optimal solutions even in convection-dominated regimes.

\subsection{Convergence Test (Boundary Layer)}\label{subsec: convergence_tests_boundary}

To further assess the accuracy and robustness of the proposed monotone scheme, we next examine its convergence behavior in convection-dominated regimes. In particular, we consider an example that exhibits pronounced boundary layers near the outflow boundaries. This test allows us to evaluate how well the EAFE method captures sharp boundary layers while maintaining stable and physically consistent approximations. The setup of the boundary-layer problem follows standard benchmarks used in \cite{leykekhman2012local, liu2024multigrid}, as described below.

\begin{example}[Boundary Layer]\label{ex:bdlayer}
  We consider the problem \eqref{eq:general} with $\Omega=(0,1)^2$, $\gamma=1$, $\bm{\zeta}=[-\sqrt{2}/2,-\sqrt{2}/2]^t$, and $\varepsilon>0$. 
  The exact solutions are prescribed as $y=\eta(x_1)\eta(x_2)$ and $p=\eta(1-x_1)\eta(1-x_2)$, where
  \[
    \eta(z)=z^3-\frac{e^{\frac{z-1}{\varepsilon}}-e^{-1/\varepsilon}}{1-e^{-1/\varepsilon}}.
  \]
  The state $y$ develops boundary layers near $x_1=1$ and $x_2=1$ for small $\varepsilon$, whereas the adjoint state $p$ (and hence the control $u$) exhibits layers near $x_1=0$ and $x_2=0$.
\end{example}

We report the numerical results in Tables~\ref{table:bdlayerg}–\ref{table:bdlayergl1}.  
For $\varepsilon = 10^{-2}$, as shown in Table~\ref{table:bdlayerg}, the global convergence orders in both the $L^2$ and $H^1$ norms increase and approach the optimal ones as $h$ decreases, which partially agrees with Theorem~\ref{thm:main} regarding $H^1$ convergence.  
  
To further examine the influence of the boundary layer, we also evaluate the local convergence orders on the subdomain $[0.4, 0.6]^2$ for $\varepsilon=10^{-2}$.  
As shown in Table~\ref{table:bdlayergl}, the local convergence orders behave similarly to the global ones, indicating that the boundary layer propagates into the interior and affects regions where the solution is otherwise smooth.  
This observation is consistent with the analysis in \cite{leykekhman2012local, heinkenschloss2010local}, which demonstrates that strongly imposed Dirichlet boundary conditions cause small oscillations in the boundary layer to propagate into the interior domain.  

We next consider the case $\varepsilon=10^{-9}$.  
From Table~\ref{table:bdlayerg1}, we observe that the global convergence order in the $L^2$ norm remains of order $O(h)$, while the $H^1$-error does not converge.  
This behavior results from the extremely sharp boundary layers present in the solution, consistent with the findings in \cite{liu2024multigrid}. 
It is worth noting that the constant appearing in the convergence estimate in Theorem~\ref{thm:main} depends on the diffusion coefficient $\varepsilon$.  
Due to this dependence, and because the proof of convergence for the EAFE method assumes sufficiently small mesh size $h$, achieving the optimal convergence order in highly convection-dominated regimes requires finer meshes as $\varepsilon$ decreases (see \cite{adler2023stable} for related numerical results).  
In practice, applying adaptive mesh refinement strategies concentrated near the boundary layers would allow one to recover optimal accuracy even for very small diffusion coefficients. This observation also suggests a natural direction for future research; developing adaptive monotone schemes that integrate local error indicators or residual-based refinement criteria while maintaining the discrete maximum principle.

Similarly, Table~\ref{table:bdlayergl1} shows that the local convergence orders do not improve significantly, again due to the strongly enforced boundary conditions.  
This further supports the discussion above, where achieving optimal convergence requires either smaller mesh sizes or adaptive refinement concentrated near boundary layers.  
Nevertheless, one of the main advantages of the EAFE method is its ability to accurately capture boundary layers, as illustrated in Figure~\ref{fig:bdlayernumer}.  
In contrast, discontinuous Galerkin methods, while often achieving optimal convergence orders in the interior, tend to underresolve or smooth out extremely thin boundary layers, especially on coarse meshes, due to their weak enforcement of boundary conditions and inherent upwind stabilization.

\begin{figure}[ht]
    \centering
    \subfloat[Numerical solution $p_h$]{\includegraphics[height=2in]{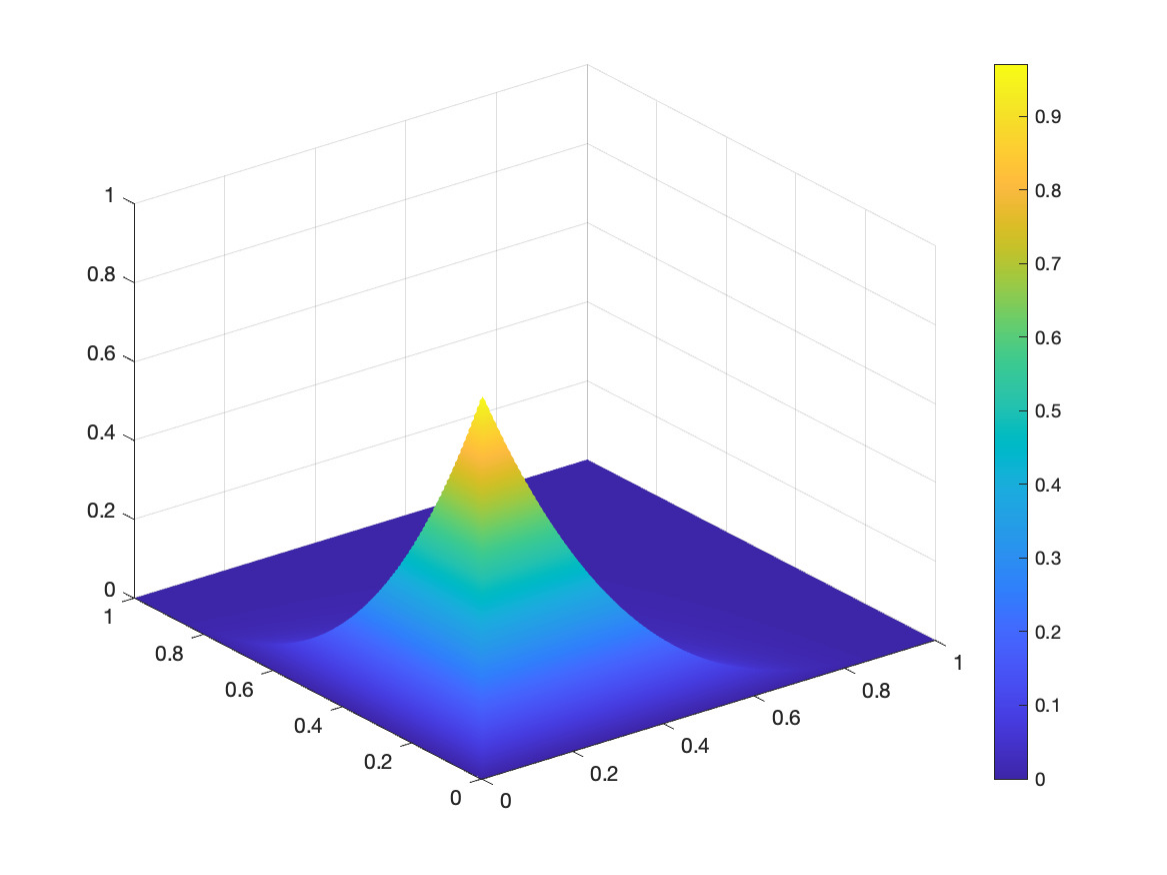}}
    \subfloat[Exact solution $p$]{\includegraphics[height=2in]{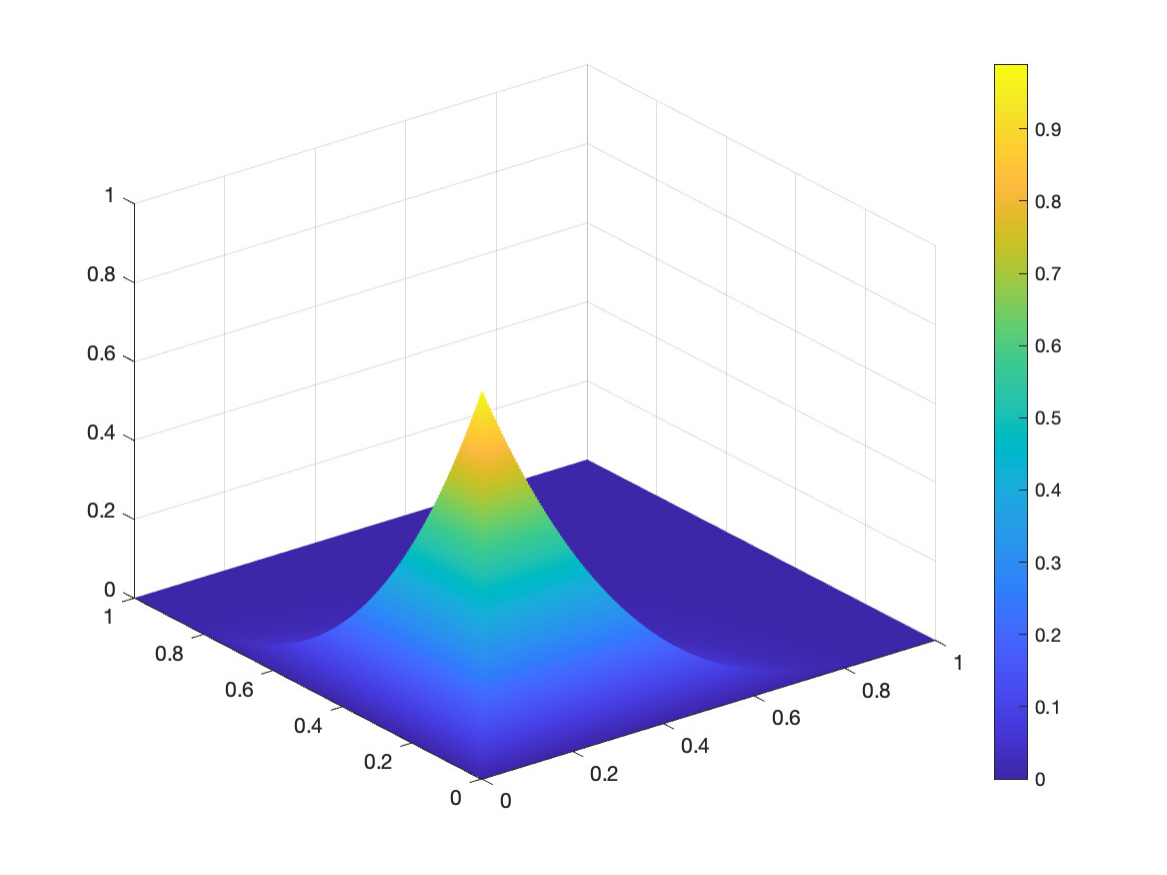}} 
    \caption{Numerical solution and exact solution for Example \ref{ex:bdlayer} with $\eps=10^{-9}$}\label{fig:bdlayernumer}
\end{figure}

\begin{table}[h!]
\centering
\footnotesize
\caption{Convergence rates for Example \ref{ex:bdlayer} with $\varepsilon=10^{-2}$ (Global)}\label{table:bdlayerg}
\begin{tabular}{ccccccccc}\hline
\\
$k$&$\|e_y\|_\LT$&Order&$\|e_y\|_{\HO}$&Order&$\|e_p\|_\LT$&Order&$\|e_p\|_{\HO}$&Order\\
\\
\hline
$1$&1.01e-02&-&1.09e-01&-&1.34e-02&-&1.06e-01&-\\
$2$&1.63e-02&-0.69&3.27e-01&-1.58&2.11e-02&-0.66&3.11e-01&-1.55\\
$3$&1.41e-02&0.21&5.37e-01&-0.72&1.65e-02&0.36&5.23e-01&-0.75\\
$4$&9.50e-03&0.57&6.90e-01&-0.36&1.01e-02&0.71&6.82e-01&-0.38\\
$5$&4.65e-03&1.03&5.55e-01&0.31&4.75e-03&1.09&5.52e-01&0.31\\
$6$&1.72e-03&1.43&3.09e-01&0.84&1.74e-03&1.45&3.08e-01&0.84\\
$7$&5.28e-04&1.71&1.44e-01&1.10&5.33e-04&1.71&1.44e-01&1.10\\
$8$&1.47e-04&1.85&6.37e-02&1.18&1.48e-04&1.85&6.37e-02&1.18\\
\hline
\end{tabular}
\end{table}

\begin{table}[h!]
\centering
\footnotesize
\caption{Convergence rates for Example \ref{ex:bdlayer} with $\varepsilon=10^{-2}$ (Local)}\label{table:bdlayergl}
\begin{tabular}{ccccccccc}\hline
\\
$k$&$\|e_y\|_\LT$&Order&$\|e_y\|_{\HO}$&Order&$\|e_p\|_\LT$&Order&$\|e_p\|_{\HO}$&Order\\
\\
\hline
$1$&4.34e-03&-&3.01e-02&-&1.21e-02&-&8.38e-02&-\\
$2$&2.95e-04&3.88&4.09e-03&2.88&4.40e-03&1.46&6.10e-02&0.46\\
$3$&1.23e-03&-2.05&3.36e-02&-3.04&1.66e-03&1.40&3.57e-02&0.77\\
$4$&7.70e-04&0.67&2.53e-02&0.41&9.30e-04&0.84&2.26e-02&0.66\\
$5$&2.87e-04&1.42&1.09e-02&1.21&3.27e-04&1.51&1.40e-02&0.69\\
$6$&8.76e-05&1.71&5.18e-03&1.08&9.00e-05&1.86&6.00e-03&1.22\\
$7$&2.48e-05&1.82&2.69e-03&0.94&2.40e-05&1.91&2.27e-03&1.40\\
$8$&6.61e-06&1.91&1.19e-03&1.18&6.29e-06&1.93&1.06e-03&1.11\\
\hline
\end{tabular}
\end{table}

\begin{table}[h!]
\centering
\footnotesize
\caption{Convergence rates for Example \ref{ex:bdlayer} with $\varepsilon=10^{-9}$ (Global)}\label{table:bdlayerg1}
\begin{tabular}{ccccccccc}\hline
\\
$k$&$\|e_y\|_\LT$&Order&$\|e_y\|_{\HO}$&Order&$\|e_p\|_\LT$&Order&$\|e_p\|_{\HO}$&Order\\
\\
\hline
$1$&8.61e-03&-&9.36e-02&-&1.46e-02&-&9.73e-02&-\\
$2$&1.52e-02&-0.82&2.66e-01&-1.51&2.34e-02&-0.68&2.47e-01&-1.34\\
$3$&1.36e-02&0.16&3.39e-01&-0.35&1.93e-02&0.28&3.22e-01&-0.38\\
$4$&9.15e-03&0.57&3.53e-01&-0.06&1.23e-02&0.64&3.42e-01&-0.09\\
$5$&5.32e-03&0.78&3.50e-01&0.02&6.98e-03&0.82&3.43e-01&-0.01\\
$6$&2.87e-03&0.89&3.44e-01&0.02&3.72e-03&0.91&3.41e-01&0.01\\
$7$&1.49e-03&0.95&3.40e-01&0.02&1.92e-03&0.96&3.39e-01&0.01\\
$8$&7.59e-04&0.97&3.38e-01&0.01&9.73e-04&0.98&3.37e-01&0.01\\
\hline
\end{tabular}
\end{table}

\begin{table}[h!]
\centering
\footnotesize
\caption{Convergence rates for Example \ref{ex:bdlayer} with $\varepsilon=10^{-9}$ (Local)}\label{table:bdlayergl1}
\begin{tabular}{ccccccccc}\hline
\\
$k$&$\|e_y\|_\LT$&Order&$\|e_y\|_{\HO}$&Order&$\|e_p\|_\LT$&Order&$\|e_p\|_{\HO}$&Order\\
\\
\hline
$1$&4.16e-03&-&2.88e-02&-&1.37e-02&-&9.47e-02&-\\
$2$&1.38e-03&1.59&1.91e-02&0.59&5.37e-03&1.35&7.45e-02&0.35\\
$3$&2.70e-03&-0.97&6.68e-02&-1.80&2.58e-03&1.06&5.56e-02&0.42\\
$4$&2.13e-03&0.34&6.20e-02&0.11&1.92e-03&0.42&4.69e-02&0.25\\
$5$&1.20e-03&0.83&4.15e-02&0.58&1.06e-03&0.86&4.17e-02&0.17\\
$6$&6.13e-04&0.97&3.14e-02&0.40&5.19e-04&1.03&3.08e-02&0.44\\
$7$&3.17e-04&0.95&2.76e-02&0.19&2.62e-04&0.99&2.13e-02&0.53\\
$8$&1.62e-04&0.97&2.22e-02&0.32&1.34e-04&0.97&1.83e-02&0.22\\
\hline
\end{tabular}
\end{table}

\subsection{Convergence Test (Interior Layer)}\label{subsec: convergence_tests_interior}

We next examine an example featuring an interior layer to further evaluate the robustness and accuracy of the EAFE method in convection-dominated regimes.

\begin{example}[Interior Layer]\label{ex:inlayer}
    In this example, we take $\Omega = (0,1)^2$, $\gamma = 1$, and $\bm{\zeta} = [-1,0]^t$, and let the exact solutions of \eqref{eq:general} be 
    \begin{equation}
        y = (1 - x_1)^3 \arctan\!\left(\frac{x_2 - 0.5}{\varepsilon}\right), 
        \qquad 
        p = x_1(1 - x_1)x_2(1 - x_2).
    \end{equation}
    The exact state $y$ exhibits an interior layer along the line $x_2 = 0.5$ when $\varepsilon$ is small.
\end{example}

The convergence orders in Example~\ref{ex:inlayer} are similar to those observed in Example~\ref{ex:bdlayer}. 
Here, we measure the local convergence orders on $[0.65,1] \times [0,1]$. 
We observe that the interior layer also affects the convergence behavior in regions away from the layer, reducing local accuracy even in smooth portions of the domain.  
Nevertheless, as shown in Figure~\ref{fig:inlayernumer}, the EAFE method accurately captures the sharp interior layer of $y$ along $x_2 = 0.5$, demonstrating its robustness in resolving steep gradients without introducing spurious oscillations.

\begin{figure}[h!]
    \centering
    \subfloat[Numerical solution $y_h$]{\includegraphics[height=2in]{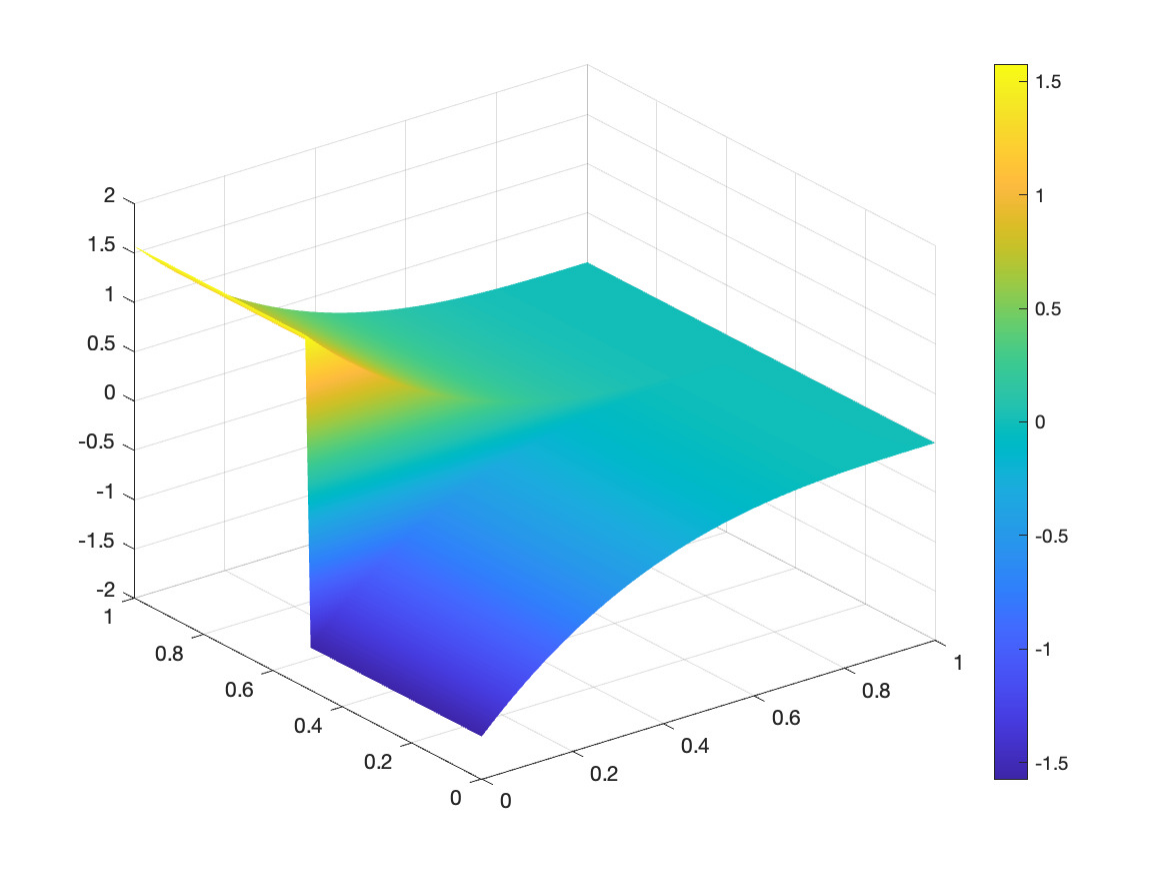}}
    \subfloat[Exact solution $y$]{\includegraphics[height=2in]{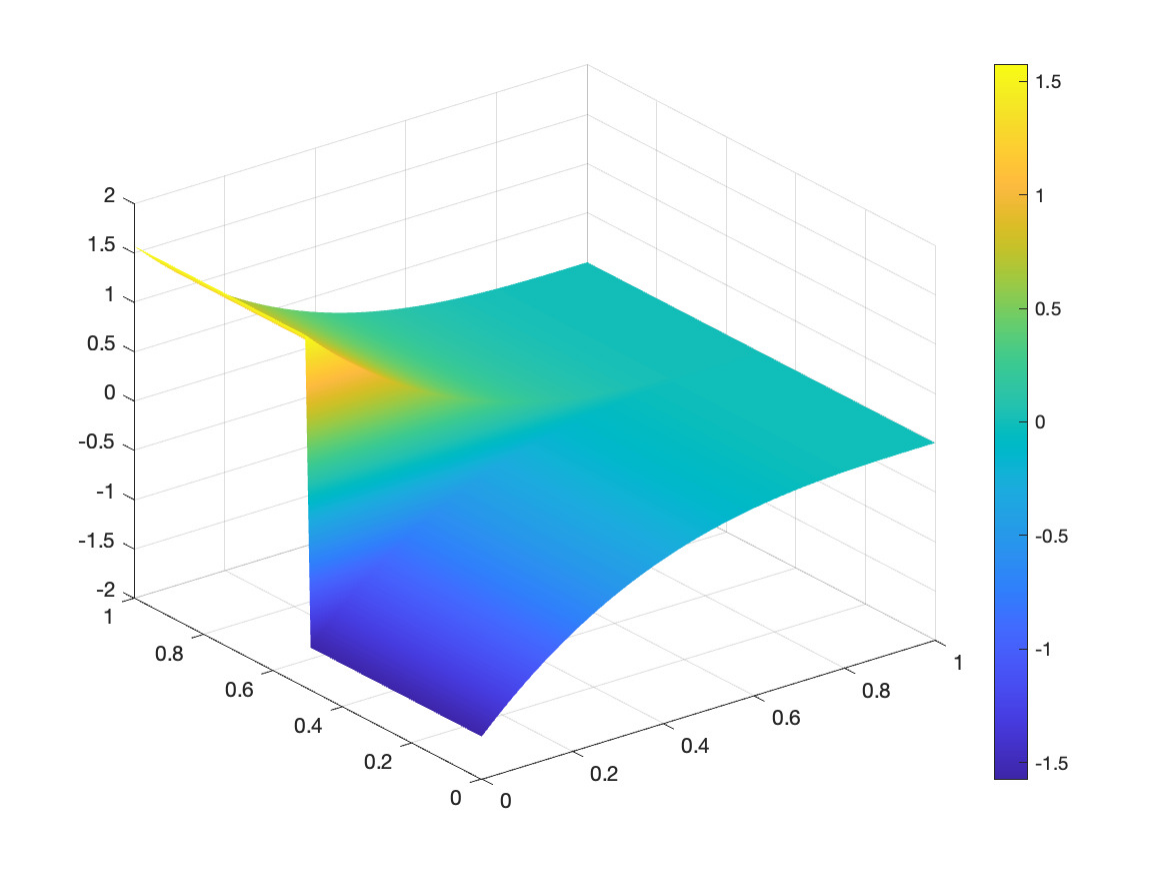}} 
    \caption{Numerical solution and exact solution for Example \ref{ex:inlayer} with $\eps=10^{-9}$}\label{fig:inlayernumer}
\end{figure}

\begin{table}[h!]
\centering
\footnotesize
\caption{Convergence rates for Example \ref{ex:inlayer} with $\varepsilon=10^{-2}$ (Global)}\label{table:inlayerg}
\begin{tabular}{ccccccccc}\hline
\\
$k$&$\|e_y\|_\LT$&Order&$\|e_y\|_{\HO}$&Order&$\|e_p\|_\LT$&Order&$\|e_p\|_{\HO}$&Order\\
\\
\hline
$1$&2.78e-03&-&1.44e-02&-&1.82e-02&-&9.38e-02&-\\
$2$&3.37e-03&-0.28&1.83e-02&-0.34&1.30e-02&0.49&6.72e-02&0.48\\
$3$&2.16e-03&0.64&1.40e-02&0.38&6.60e-03&0.98&4.60e-02&0.55\\
$4$&9.63e-04&1.17&8.03e-03&0.81&2.67e-03&1.31&3.47e-02&0.40\\
$5$&3.36e-04&1.52&3.40e-03&1.24&9.08e-04&1.55&2.41e-02&0.53\\
$6$&1.02e-04&1.72&1.14e-03&1.58&2.77e-04&1.71&1.37e-02&0.82\\
$7$&2.86e-05&1.83&3.32e-04&1.77&7.83e-05&1.82&7.01e-03&0.96\\
$8$&7.64e-06&1.91&9.04e-05&1.88&2.10e-05&1.90&3.50e-03&1.00\\
\hline
\end{tabular}
\end{table}

\begin{table}[h!]
\centering
\footnotesize
\caption{Convergence rates for Example \ref{ex:inlayer} with $\varepsilon=10^{-2}$ (Local)}\label{table:inlayergl}
\begin{tabular}{ccccccccc}\hline
\\
$k$&$\|e_y\|_\LT$&Order&$\|e_y\|_{\HO}$&Order&$\|e_p\|_\LT$&Order&$\|e_p\|_{\HO}$&Order\\
\\
\hline
$1$&1.36e-03&-&1.33e-02&-&6.77e-03&-&6.63e-02&-\\
$2$&1.97e-03&-0.54&1.93e-02&-0.54&4.52e-03&0.58&4.48e-02&0.56\\
$3$&1.45e-03&0.44&1.68e-02&0.20&2.57e-03&0.81&3.98e-02&0.17\\
$4$&7.15e-04&1.02&1.02e-02&0.73&1.22e-03&1.07&2.96e-02&0.43\\
$5$&2.60e-04&1.46&4.63e-03&1.13&4.73e-04&1.37&1.86e-02&0.67\\
$6$&8.02e-05&1.69&1.77e-03&1.38&1.57e-04&1.59&1.03e-02&0.85\\
$7$&2.28e-05&1.81&6.33e-04&1.49&4.70e-05&1.74&5.28e-03&0.96\\
$8$&6.14e-06&1.89&2.22e-04&1.51&1.30e-05&1.85&2.60e-03&1.02\\
\hline
\end{tabular}
\end{table}

\begin{table}[h!]
\centering
\footnotesize
\caption{Convergence rates for Example \ref{ex:inlayer} with $\varepsilon=10^{-9}$ (Global)}\label{table:inlayerg1}
\begin{tabular}{ccccccccc}\hline
\\
$k$&$\|e_y\|_\LT$&Order&$\|e_y\|_{\HO}$&Order&$\|e_p\|_\LT$&Order&$\|e_p\|_{\HO}$&Order\\
\\
\hline
$1$&3.05e-03&-&1.57e-02&-&1.91e-02&-&9.81e-02&-\\
$2$&3.93e-03&-0.37&2.11e-02&-0.43&1.46e-02&0.39&7.50e-02&0.39\\
$3$&2.93e-03&0.42&1.85e-02&0.19&8.56e-03&0.77&5.55e-02&0.43\\
$4$&1.77e-03&0.73&1.41e-02&0.39&4.58e-03&0.90&4.47e-02&0.31\\
$5$&9.69e-04&0.87&1.03e-02&0.46&2.37e-03&0.95&3.96e-02&0.18\\
$6$&5.08e-04&0.93&7.37e-03&0.48&1.20e-03&0.98&3.73e-02&0.09\\
$7$&2.60e-04&0.97&5.24e-03&0.49&6.06e-04&0.99&3.62e-02&0.04\\
$8$&1.32e-04&0.98&3.72e-03&0.50&3.04e-04&0.99&3.56e-02&0.02\\
\hline
\end{tabular}
\end{table}

\begin{table}[h!]
\centering
\footnotesize
\caption{Convergence rates for Example \ref{ex:inlayer} with $\varepsilon=10^{-9}$ (Local)}\label{table:inlayergl1}
\begin{tabular}{ccccccccc}\hline
\\
$k$&$\|e_y\|_\LT$&Order&$\|e_y\|_{\HO}$&Order&$\|e_p\|_\LT$&Order&$\|e_p\|_{\HO}$&Order\\
\\
\hline
$1$&1.47e-03&-&1.44e-02&-&6.96e-03&-&6.82e-02&-\\
$2$&2.27e-03&-0.63&2.23e-02&-0.63&4.84e-03&0.52&4.75e-02&0.52\\
$3$&1.93e-03&0.24&2.22e-02&0.00&2.94e-03&0.72&4.37e-02&0.12\\
$4$&1.26e-03&0.61&1.82e-02&0.29&1.63e-03&0.85&3.55e-02&0.30\\
$5$&7.07e-04&0.84&1.38e-02&0.40&8.28e-04&0.98&2.90e-02&0.29\\
$6$&3.71e-04&0.93&1.01e-02&0.45&4.13e-04&1.01&2.68e-02&0.12\\
$7$&1.91e-04&0.96&7.26e-03&0.47&2.07e-04&0.99&2.60e-02&0.04\\
$8$&9.71e-05&0.98&5.18e-03&0.49&1.04e-04&0.99&2.53e-02&0.04\\
\hline
\end{tabular}
\end{table}


\section{Conclusion}
\label{sec:conclusion}

In this work, we developed and analyzed a monotone finite element scheme for elliptic distributed optimal control problems constrained by a convection-diffusion-reaction equation. 
By employing the EAFE discretization, we established a framework that not only ensures monotonicity but also preserves the desired-state bounds in the discrete setting. 
This property guarantees that the numerical optimal state remains stable and free from nonphysical oscillations, even in convection-dominated regimes. 

The theoretical analysis combined the consistency of the EAFE scheme with a discrete inf-sup condition, leading to a rigorous proof of well-posedness and an $O(h)$ convergence order in the $H^1$ norm. 
Extensive numerical experiments further validated the theoretical results, confirming that the EAFE method accurately captures both boundary and interior layers without introducing spurious oscillations.

Future research directions include extending the present framework to adaptive mesh refinement strategies (e.g., \cite{yan2025stabilization,yucel2015discontinuous}) and preconditioning techniques (e.g., \cite{jeong2025c,brenner2023additive,liu2025balancing,liu2025convergence}) specifically tailored for the EAFE method in optimal control problems. It is also interesting to consider pointwise control constraints \cite{BOMMANABOYANA2025100624,becker2007optimal} and state constraints \cite{brenner2023multigrid,liu2024discontinuous,brenner2024c0}.
Another promising direction is to generalize the approach to more complex multiphysics optimal control settings, such as those involving coupled flow and transport phenomena or high-dimensional convection-diffusion systems.

\subsection*{Acknowledgment}

This material is based upon work supported by the National Science Foundation under Grant No. DMS-1929284 while S. Liu was in residence at the Institute for Computational and Experimental Research in Mathematics in Providence, RI, during the Numerical PDEs: Analysis, Algorithms, and Data Challenges semester program.


\bibliographystyle{plain}
\bibliography{EAFE_OCP}

\begin{thebibliography}{10}

\bibitem{adler2023stable}
James~H. Adler, Casey Cavanaugh, Xiaozhe Hu, Andy Huang, and Nathaniel Trask.
\newblock A stable mimetic finite-difference method for convection-dominated
  diffusion equations.
\newblock {\em SIAM Journal on Scientific Computing}, 45(6):A2973--A3000, 2023.

\bibitem{arnold1982interior}
Douglas~N. Arnold.
\newblock An interior penalty finite element method with discontinuous
  elements.
\newblock {\em SIAM Journal on Numerical Analysis}, 19(4):742--760, 1982.

\bibitem{arnold2002unified}
Douglas~N. Arnold, Franco Brezzi, Bernardo Cockburn, and L.~Donatella Marini.
\newblock Unified analysis of discontinuous {G}alerkin methods for elliptic
  problems.
\newblock {\em SIAM Journal on Numerical Analysis}, 39(5):1749--1779, 2002.

\bibitem{ayuso2009discontinuous}
Blanca Ayuso and L.~Donatella Marini.
\newblock Discontinuous {G}alerkin methods for advection-diffusion-reaction
  problems.
\newblock {\em SIAM Journal on Numerical Analysis}, 47(2):1391--1420, 2009.

\bibitem{Bab}
Ivo Babu{\v{s}}ka.
\newblock The finite element method with {L}agrangian multipliers.
\newblock {\em Numerische Mathematik}, 20(3):179--192, 1973.

\bibitem{bacuta2025convergence}
Constantin Bacuta.
\newblock On convergence of upwinding {P}etrov-{G}alerkin methods for
  convection-diffusion.
\newblock {\em arXiv preprint arXiv:2509.04703}, 2025.

\bibitem{barrenechea2018unified}
Gabriel~R. Barrenechea, Volker John, Petr Knobloch, and Richard Rankin.
\newblock A unified analysis of algebraic flux correction schemes for
  convection-diffusion equations.
\newblock {\em SeMA Journal}, 75(4):655--685, 2018.

\bibitem{baumgartner2025afc}
Jens Baumgartner and Arnd R{\"o}sch.
\newblock {AFC} stabilization of a state constrained optimal control problem.
\newblock {\em Mathematical Control and Related Fields}, pages 0--0, 2025.

\bibitem{becker2007optimal}
Roland Becker and Boris Vexler.
\newblock Optimal control of the convection-diffusion equation using stabilized
  finite element methods.
\newblock {\em Numerische Mathematik}, 106:349--367, 2007.

\bibitem{da2021supg}
Lourenco Beir\~ao~da Veiga, Franco Dassi, Carlo Lovadina, and Giuseppe Vacca.
\newblock {SUPG}-stabilized virtual elements for diffusion-convection problems:
  a robustness analysis.
\newblock {\em ESAIM: Mathematical Modelling and Numerical Analysis},
  55(5):2233--2258, 2021.

\bibitem{BOMMANABOYANA2025100624}
Satyajith {Bommana Boyana}, Thomas Lewis, Sijing Liu, and Yi~Zhang.
\newblock Convergence analysis of a dual-wind discontinuous galerkin method for
  an elliptic optimal control problem with control constraints.
\newblock {\em Results in Applied Mathematics}, 27:100624, 2025.

\bibitem{brenner2024c0}
Susanne~C. Brenner, SeongHee Jeong, Li-yeng Sung, and Zhiyu Tan.
\newblock ${C}^0$ interior penalty methods for an elliptic distributed optimal
  control problem with general tracking and pointwise state constraints.
\newblock {\em Computers \& Mathematics with Applications}, 155:80--90, 2024.

\bibitem{brenner2020multigrid}
Susanne~C. Brenner, Sijing Liu, and Li-yeng Sung.
\newblock Multigrid methods for saddle point problems: {O}ptimality systems.
\newblock {\em Journal of Computational and Applied Mathematics}, 372:112733,
  2020.

\bibitem{brenner2021p1}
Susanne~C. Brenner, Sijing Liu, and Li-yeng Sung.
\newblock A ${P}_1$ finite element method for a distributed elliptic optimal
  control problem with a general state equation and pointwise state
  constraints.
\newblock {\em Computational Methods in Applied Mathematics}, 21(4):777--790,
  2021.

\bibitem{brenner2023multigrid}
Susanne~C Brenner, Sijing Liu, and Li-yeng Sung.
\newblock Multigrid methods for an elliptic optimal control problem with
  pointwise state constraints.
\newblock {\em Results in Applied Mathematics}, 17:100356, 2023.

\bibitem{BS}
Susanne~C. Brenner and L.~Ridgway Scott.
\newblock {\em The Mathematical Theory of Finite Element Methods}, volume~15.
\newblock Springer Science \& Business Media, 2008.

\bibitem{brenner2023additive}
Susanne~C. Brenner, Li-yeng Sung, and Kening Wang.
\newblock Additive {S}chwarz preconditioners for ${C}^0$ interior penalty
  methods for a state constrained elliptic distributed optimal control problem.
\newblock In {\em Domain Decomposition Methods in Science and Engineering
  XXVI}, pages 607--615. Springer, 2023.

\bibitem{Brezzi}
Franco Brezzi.
\newblock On the existence, uniqueness and approximation of saddle-point
  problems arising from {L}agrangian multipliers.
\newblock {\em Revue fran{\c{c}}aise d'automatique, informatique, recherche
  op{\'e}rationnelle. Analyse num{\'e}rique}, 8(R2):129--151, 1974.

\bibitem{brooks1982streamline}
Alexander~N. Brooks and Thomas~J.R. Hughes.
\newblock Streamline upwind/{P}etrov-{G}alerkin formulations for convection
  dominated flows with particular emphasis on the incompressible
  {N}avier-{S}tokes equations.
\newblock {\em Computer Methods in Applied Mechanics and Engineering},
  32(1-3):199--259, 1982.

\bibitem{cao2025edge}
Shuhao Cao, Long Chen, and Seulip Lee.
\newblock Edge-averaged virtual element methods for convection-diffusion and
  convection-dominated problems.
\newblock {\em Journal of Scientific Computing}, 104(2):1--24, 2025.

\bibitem{chen2008ifem}
Long Chen.
\newblock {$i$FEM: An Integrated Finite Element Methods Package in MATLAB}.
\newblock {\em Technical Report, University of California at Irvine}, 2009.

\bibitem{collis2002analysis}
S.~Scott Collis and Matthias Heinkenschloss.
\newblock Analysis of the streamline upwind/{P}etrov {G}alerkin method applied
  to the solution of optimal control problems.
\newblock {\em CAAM TR02-01}, 108, 2002.

\bibitem{evans10}
Lawrence~C. Evans.
\newblock {\em Partial Differential Equations}.
\newblock American Mathematical Society, Providence, R.I., 2010.

\bibitem{heinkenschloss2010local}
Matthias Heinkenschloss and Dmitriy Leykekhman.
\newblock Local error estimates for {SUPG} solutions of advection-dominated
  elliptic linear-quadratic optimal control problems.
\newblock {\em SIAM Journal on Numerical Analysis}, 47(6):4607--4638, 2010.

\bibitem{jeong2025optimal}
SeongHee Jeong and Sanghyun Lee.
\newblock Optimal control for {D}arcy's equation in a heterogeneous porous
  media.
\newblock {\em Applied Numerical Mathematics}, 207:303--322, 2025.

\bibitem{jeong2025c}
SeongHee Jeong, Seulip Lee, and Kening Wang.
\newblock A ${C}^{0}$ weak {G}alerkin method with preconditioning for
  constrained optimal control problems with general tracking.
\newblock {\em arXiv preprint arXiv:2506.17619}, 2025.

\bibitem{johnson2012numerical}
Claes Johnson.
\newblock {\em Numerical Solution of Partial Differential Equations by the
  Finite Element Method}.
\newblock Courier Corporation, 2009.

\bibitem{knabner2003numerical}
Peter Knabner and Lutz Angermann.
\newblock {\em Numerical Methods for Elliptic and Parabolic Partial
  Differential Equations}.
\newblock Springer, 2003.

\bibitem{knobloch2009local}
Petr Knobloch and Gert Lube.
\newblock Local projection stabilization for advection-diffusion-reaction
  problems: {O}ne-level vs. two-level approach.
\newblock {\em Applied Numerical Mathematics}, 59(12):2891--2907, 2009.

\bibitem{lesaint1974finite}
Pierre Lesaint and Pierre-Arnaud Raviart.
\newblock On a finite element method for solving the neutron transport
  equation.
\newblock {\em Publications des s{\'e}minaires de math{\'e}matiques et
  informatique de Rennes}, (S4):1--40, 1974.

\bibitem{leykekhman2012local}
Dmitriy Leykekhman and Matthias Heinkenschloss.
\newblock Local error analysis of discontinuous {G}alerkin methods for
  advection-dominated elliptic linear-quadratic optimal control problems.
\newblock {\em SIAM Journal on Numerical Analysis}, 50(4):2012--2038, 2012.

\bibitem{li2021local}
Yang Li and Minfu Feng.
\newblock A local projection stabilization virtual element method for
  convection-diffusion-reaction equation.
\newblock {\em Applied Mathematics and Computation}, 411:126536, 2021.

\bibitem{Lions}
Jacques~Louis Lions.
\newblock {\em Optimal Control of Systems Governed by Partial Differential
  Equations}.
\newblock Springer, 1971.

\bibitem{liu2025robust}
Sijing Liu.
\newblock Robust multigrid methods for discontinuous galerkin discretizations
  of an elliptic optimal control problem.
\newblock {\em Computational Methods in Applied Mathematics}, 25(1):133--151,
  2025.

\bibitem{liu2024multigrid}
Sijing Liu and Valeria Simoncini.
\newblock Multigrid preconditioning for discontinuous {G}alerkin
  discretizations of an elliptic optimal control problem with a
  convection-dominated state equation.
\newblock {\em Journal of Scientific Computing}, 101(3):79, 2024.

\bibitem{liu2024discontinuous}
Sijing Liu, Zhiyu Tan, and Yi~Zhang.
\newblock Discontinuous galerkin methods for an elliptic optimal control
  problem with a general state equation and pointwise state constraints.
\newblock {\em Journal of Computational and Applied Mathematics}, 437:115494,
  2024.

\bibitem{liu2025balancing}
Sijing Liu and Jinjin Zhang.
\newblock A balancing domain decomposition by constraints preconditioner for a
  hybridizable discontinuous {G}alerkin discretization of an elliptic optimal
  control problem.
\newblock {\em arXiv preprint arXiv:2504.02072}, 2025.

\bibitem{liu2025convergence}
Sijing Liu and Jinjin Zhang.
\newblock Convergence analysis of a balancing domain decomposition method for
  an elliptic optimal control problem with {HDG} discretizations.
\newblock {\em arXiv preprint arXiv:2508.13997}, 2025.

\bibitem{morton1995numerical}
Keith~William Morton.
\newblock {\em Numerical solution of convection-diffusion problems}.
\newblock Applied Mathematics. Springer Netherlands, 1995.

\bibitem{rooscdr}
Hans-G{\"o}rg Roos, Martin Stynes, and Lutz Tobiska.
\newblock {\em Robust Numerical Methods for Singularly Perturbed Differential
  Equations: Convection-Diffusion-Reaction and Flow Problems}.
\newblock Springer, 2008.

\bibitem{Tro}
Fredi Tr{\"o}ltzsch.
\newblock {\em Optimal Control of Partial Differential Equations: Theory,
  Methods, and Applications}, volume 112.
\newblock American Mathematical Society, 2010.

\bibitem{xu1999monotone}
Jinchao Xu and Ludmil~T. Zikatanov.
\newblock A monotone finite element scheme for convection-diffusion equations.
\newblock {\em Mathematics of Computation}, 68(228):1429--1446, 1999.

\bibitem{yan2025stabilization}
Qiuhui Yan, Xufeng Xiao, and Xinlong Feng.
\newblock Stabilization and adaptive {FEM} for optimal control problems of
  stationary convection-dominated diffusion equations on surfaces.
\newblock {\em Applied Numerical Mathematics}, 2025.

\bibitem{yucel2015discontinuous}
Hamdullah Y{\"u}cel, Martin Stoll, and Peter Benner.
\newblock A discontinuous {G}alerkin method for optimal control problems
  governed by a system of convection-diffusion {PDE}s with nonlinear reaction
  terms.
\newblock {\em Computers \& Mathematics with Applications}, 70(10):2414--2431,
  2015.

\end{thebibliography}
	
\end{document}